\documentclass[12pt]{amsart}
\usepackage{Baskervaldx}
\usepackage[baskervaldx,cmintegrals,bigdelims,vvarbb]{newtxmath}
\usepackage[cal=cm]{mathalfa}
\usepackage{enumitem, pgfplots, graphicx, subcaption, amsmath, array, pdfsync, tikz, tikz-cd, color, url, float, fullpage, wrapfig,comment, scrextend, wrapfig, amsthm}

\usetikzlibrary{decorations.pathreplacing,calligraphy}

\usepackage[normalem]{ulem}

\usepackage[font=footnotesize]{caption}

\usetikzlibrary{calc}

\title{Smooth limits of plane curves of prime degree and Markov numbers\vspace{-2ex}}
\author{Kristin DeVleming and David Stapleton\vspace{-2ex}}

\usepackage[pagebackref,pdfencoding=auto]{hyperref}
\hypersetup{
  colorlinks= true,
  linkcolor=[RGB]{0,89,42},
  urlcolor=[RGB]{0,89,42},
  citecolor=[RGB]{0,89,42}
}

\definecolor{mycolor}{RGB}{146, 214, 203}
\definecolor{myothercolor}{RGB}{179, 215, 232}

\definecolor{mycolor}{RGB}{146, 214, 203}
\definecolor{myothercolor}{RGB}{179, 215, 232}

\newtheorem{theorem}{Theorem}
\newtheorem{corollary}[theorem]{Corollary}
\newtheorem{lemma}[theorem]{Lemma}
\newtheorem{proposition}[theorem]{Proposition}

\numberwithin{theorem}{section}

\newtheorem*{blankcorollary}{Corollary}

\newtheorem{manualtheoreminner}{Theorem}
\newenvironment{manualtheorem}[1]{%
  \renewcommand\themanualtheoreminner{#1}%
  \manualtheoreminner
}{\endmanualtheoreminner}

\newtheorem{manualconjinner}{Conjecture}
\newenvironment{manualconj}[1]{%
  \renewcommand\themanualconjinner{#1}%
  \manualconjinner
}{\endmanualconjinner}

\theoremstyle{definition}

\theoremstyle{definition}
\newtheorem{remark}[theorem]{Remark}

\newtheorem{example}[theorem]{Example}
\newtheorem{definition}[theorem]{Definition}

\newcommand{\hr}[2]{\hyperref[#1]{#2}}

\def\ZZ{{\mathbb Z}}
\def\FF{{\mathbb F}}
\def\QQ{{\mathbb Q}}

\def\AA{{\mathbb A}}
\def\NN{{\mathbb N}}
\def\CC{{\mathbb C}}

\def\PP{{\mathbf{P}}}

\def\Lc{{\mathcal L}}
\def\Oc{{\mathcal O}}
\def\Cc{{\mathcal C}}

\def\Mc{{\mathcal M}}

\def\Qc{{\mathcal Q}}

\def\hbar{{\overline{h}}}

\def\deg{{\mathrm{deg}}}

\def\Spec{{\mathrm{Spec}}}

\def\Pic{{\mathrm{Pic}}}

\def\Supp{{\mathrm{Supp}}}

\def\Xcy{{X^{\mathrm{CY}}}}
\def\Xcyo{{X^{\mathrm{CY}}_0}}
\def\Dcy{{D^{\mathrm{CY}}}}
\def\Dss{{D^{\mathrm{ss}}}}
\def\Dsso{{D^{\mathrm{ss}}_0}}
\def\Dcyo{{D^{\mathrm{CY}}_0}}
\def\Dnorm{{D^{\mathrm{norm}}}}
\def\Dnormo{{\Dnorm_0}}

\def\Dnorm{{D^{\mathrm{norm}}}}
\def\Dnormo{{D^{\mathrm{norm}}_0}}

\def\Do{{D_0}}

\def\lct{\mathrm{lct}}

\def\loc{{\mathrm{loc}}}
\def\Yt{{\widetilde{Y}}}
\def\Ct{{\widetilde{C}}}
\def\St{{\widetilde{S}}}
\def\pf{{\mathfrak{p}}}
\def\cond{{\mathrm{cond}}}

\def\Pabc{{\PP(a^2,b^2,c^2)}}

\def\Cl{{\mathrm{Cl}}}
\def\mf{{\mathfrak{m}}}
\def\wt{{\mathrm{wt}}}
\def\length{{\mathrm{length}}}
\def\red{{\mathrm{red}}}
\def\xbar{{\bar{x}}}
\def\ybar{{\bar{y}}}
\def\norm{{\mathrm{norm}}}

\def\Hom{{\mathrm{Hom}}}
\def\Ext{{\mathrm{Ext}}}

\def\dra{{\dashrightarrow}}
\def\ra{{\rightarrow}}

\def\cl{{\colon}}
\def\tdiv{{\mathrm{div}}}

\definecolor{greener}{RGB}{72,171,131}

\pgfplotsset{compat=1.15}

\begin{document}
\maketitle

\begin{center}
\textit{To Bear, Junior, and Bird.}
\end{center}

\thispagestyle{empty}

In this paper, we consider a family $D\ra T$ of smooth compact complex curves. Assuming the general fiber $D_t$ is a smooth plane curve of degree $d>1$ we ask the following:

\noindent\textbf{Question.} \textit{For which degrees $d$ can we guarantee that \textbf{every} fiber is a plane curve?}

\noindent It is easy to see that $d$ must be prime. Classically, a nonhyperelliptic genus 3 curve is a canonically embedded degree 4 plane curve, but the canonical map for a hyperelliptic genus 3 curve gives a double cover of a conic. Similarly, when $d$ is composite it is well understood how to degenerate a degree $ab$ hypersurface in any dimension to a degree $a$ branched cover of a degree $b$ hypersurface (\cite[Ex. 1.59]{KollarModuliBook}) so it is necessary to consider prime degrees.

In higher dimensions, Mori asked if being prime is also sufficient.

\noindent\textbf{Question.}(\cite[p. 642]{Moriconjecture}) \textit{If $n\ge 3$, is every smooth projective limit of prime degree hypersurfaces of dimension $n$ in $\PP^{n+1}_\CC$ also a hypersurface in $\PP^{n+1}_\CC$?}

\noindent This has been proven for the primes 2 \cite{Quadric1,Quadric2,Quadric3}, 3 \cite{Cubic}, and 5 \cite{OttemSchreieder} in all dimensions, and for the prime 7 in dimension 3 \cite{OttemSchreieder}. Interestingly, the statement is false if the dimension is 1 or 2. The purpose of this paper is to develop and provide evidence for an analogous conjecture in the case of plane curves.

Griffin gave an example \cite{Griffin} of a family of smooth plane quintics such that the limit is hyperelliptic and consequently nonplanar. We generalize this fact by proving that many prime degrees admit non-planar limits:

\begin{manualtheorem}{A}\label{thmA:nonplanarlimits}
For any Markov number $d >2$, there is family of smooth plane curves of degree $d$ with a smooth projective non-planar limit.  In particular, for any Markov number $p>2$ that is prime there is a smooth family of prime degree $p$ plane curves with a non-planar central fiber.
\end{manualtheorem}
\vspace{-5pt}

Recall that a \textit{Markov number} is a natural number that appears as a solution to the equation:
\[
a^2+b^2+c^2=3abc.
\]
The first few Markov numbers are
\[
1,2,5,13,29,34,89,...
\]
There are infinitely many Markov numbers and the Markov triples are naturally organized into a binary tree, where every Markov number is obtained from (1,1,1) by repeating a standard ``mutation" process. In this paper, the relation between Markov numbers and plane curves is the following: the only mildly singular (log terminal and $\mathbb{Q}$-Gorenstein) degenerations of $\PP^2$ are the \textit{Manetti surfaces} --- that is, either a weighted projective space $\PP(a^2,b^2,c^2)$ where $(a,b,c)$ is a Markov triple, or a partial smoothing of one of these weighted projective spaces (see \S\ref{sec:degensofP2} for more details).

Motivated by this construction and general results from moduli of stable pairs compactifying the space of pairs $(\PP^2, C)$, we conjecture the following:

\begin{manualconj}{B}\label{conjB}
Any smooth limit of a family of plane curves of prime degree is a Cartier divisor in a Manetti surface.
\end{manualconj}
\vspace{-5pt}

\noindent Conjecture \ref{conjB} implies the following.

\begin{manualconj}{C}\label{conjC}
Let $p$ be a prime number that is not a Markov number. Any smooth limit of plane curves of degree $p$ is a plane curve.
\end{manualconj}
\vspace{-5pt}

As any smooth limit of a family of curves of degree $2$ or $3$ is clearly planar, the first primes to verify the conjecture are $p =5$ and $p = 7$.  Our main result is to prove the conjectures in these cases (see \S\ref{sec:degensofP2} for the notation $M(5)$). 

\begin{manualtheorem}{D}\label{thmD:deg5}
Every smooth projective limit of a family of degree $5$ plane curves is either planar or is a Cartier divisor in the Manetti surface $M(5)$, and the smooth limits in $M(5)$ are all hyperelliptic.
\end{manualtheorem}
\vspace{-5pt}

\begin{manualtheorem}{E}\label{thmE:deg7}
Every smooth projective limit of a family of degree $7$ plane curves is a plane curve.
\end{manualtheorem}
\vspace{-5pt}

Verifying Conjectures \ref{conjB} and \ref{conjC} has strong consequences for the intersection of various loci in $M_{g}$, the moduli space of smooth genus $g := g(d)$ curves, where $g(d)$ is the genus of a smooth degree $d$ plane curve.  In particular, Conjecture \ref{conjB} places bounds on the gonality of the curves that can be in the closure of the locus of planar curves, and Conjecture \ref{conjC} implies that for $p$ prime but not a Markov number, the locus of degree $p$ plane curves
\[
P_p\subset M_g
\]
is closed. In particular, this implies that the Brill Noether locus of curves of gonality less than $p-1$ does not meet $P_p$. The following Corollaries are immediate consequences of Theorems \ref{thmD:deg5} and \ref{thmE:deg7}. The first uses that every hyperelliptic genus 6 curve can be written as a limit of a family of plane quintics (Example \ref{ex:quintics}). 

\begin{blankcorollary}
    The closure of $P_5$ in $M_6$ is $P_6\cup H_6$ (where $H_6$ is the hyperelliptic locus).
\end{blankcorollary}

\begin{blankcorollary}
    $P_7$ is closed in $M_{15}$. In particular, a plane septic cannot degenerate to a curve with gonality less than $6$. 
\end{blankcorollary}

\noindent By Remark~\ref{rem:Zagier}, Conjecture~\ref{conjC} would imply that for almost any prime, the locus of plane curves of that degree is closed in the moduli space of smooth curves.

In \S\ref{sec:degensofP2}, we study the Class group, Picard group, and deformations of (divisors on) Manetti surfaces and prove Theorem \ref{thmA:nonplanarlimits} by constructing smooth Cartier divisors in such a degeneration, bounding the gonality of the resulting curves.  

In \S\ref{sec:CYlimits} we propose a general approach to the conjecture by using Hacking's work \cite{Hacking} on limits of pairs $(\PP^2,C)$. To a family of plane curves $D\ra T$ as above (after a possible base change of $T$) there is an associated threefold pair $(\Xcy,\Dcy)\ra T$ that we call Hacking's Calabi-Yau limit (see \cite[Defn. 2.4]{Hacking}) satisfying:
\begin{enumerate}
\item for a general fiber $t\in T$ we have $X^{\mathrm{CY}}_t\cong \PP^2$ and $D_t\cong D^{\mathrm{CY}}_t$, and
\item over the central fiber $0\in T$, the fiber $\Xcyo$ is a log terminal limit of $\PP^2$ such that $(\Xcyo, \frac{3}{d} \Dcyo)$ is log canonical.
\end{enumerate}
\noindent In \S\ref{sec:CYlimits}, we provide background on these limits and list all possible $\Xcyo$ when $d = 5$ or $d = 7$.  

On the other hand, the pair $(\PP^2,D_t)$ has a limit as a KSBA stable pair, i.e. there is a threefold $(X,D)$ with a map to $T$ such that the general fiber is $(\PP^2,D_t)$ and the central fiber is an slc pair $(X_0,D_0)$. Here $D_0$ is (as the notation suggests) the original smooth central fiber. There is a MMP that interpolates between $(X,D)$ and $(\Xcy,\Dcy)$. Roughly speaking, to prove Conjecture~\ref{conjB} it would suffice to prove that nothing happens in this MMP. 

Our approach is to compare the birational surfaces $D \dra \Dcy$.  If $\Dcy$ is not normal, we may consider the birational map between $D$ and the normalization $\Dnorm$ of $\Dcy$; $D \dra \Dnorm \ra \Dcy$.  There is a common resolution $\Dss$ of
\[
D \dra \Dnorm \ra \Dcy
\]
by taking a log resolution of $\Dcy$ such that the central fiber $\Dsso$ is a semistable curve, and $D$ is obtained from $\Dss$ by contracting trees of rational curves in the central fiber.  By the assumption that $D_0$ is smooth, it follows that the dual graph of $\Dsso$ is a tree and $\Dsso$ has exactly one non-rational component.  The majority of the paper is dedicated to comparing the geometry of these surfaces, specifically for degrees $5$ and $7$. 

In \S\ref{sec:Dcynorm}, we collect general results on the map $\Dnorm \ra \Dcy$ to study how the geometry of the central fiber $\Dcyo$ is related to $\Dnormo$. Along the way we prove that any S2 variety is obtained from its normalization by \textit{codimension 1 gluing conditions}. In \S\ref{sec:dualgraphs} we define the \textit{intersection graph} $\Gamma(C)$ of an arbitrary curve $C$ which is a bipartite graph that generalizes the dual graph of a semistable curve. Our main result in \S4 is that in an S2 family, the intersection graph of the central fiber of the normalization can be controlled by the intersection graph of the central fiber and its multiplicities. This gives the following application of independent interest:

\begin{manualtheorem}{F}\label{thmF:intgraphs}
Let $\Lambda$ be a linear series on a smooth projective surface $S$ with general fiber connected and smooth. Let $C=m_1C_1+\cdots+m_\ell C_\ell$ be a possibly nonreduced curve in $\Lambda$ with multiplicities $m_1,\dots,m_\ell$. Then any semistable replacement $D'$ of $C$ satisfies:
\[
\left(\begin{array}{l}\text{\# of loops in}\\\text{dual graph of }D'\end{array}\right) \ge \left(\text{\# of loops in }\Gamma(C)\right)-\sum (m_i-1).
\]
\end{manualtheorem}
In the families we consider, the dual graph of $\Dsso$ is a tree and all but one component of $\Dsso$ are rational. It follows that the intersection graph of $\Dnormo$ is a tree and the normalization of all but one component is rational.  This places strong constraints on the possible intersection graphs for $\Dcyo$ and singularities of the components of $\Dcyo$.  After collecting necessary background on curve singularities in \S\ref{sec:curvebackground}, in \S\ref{sec:reducedlimits} and \S\ref{sec:nonreduceddeg7}, we reduce the proofs of Theorems~\ref{thmD:deg5} and \ref{thmE:deg7} to casework.  Using the possible intersection graphs and classification of low degree rational cuspidal curves, we prove that, for $d = 5, 7$, the the only possible curve configuration $\Dcyo$ with log canonical threshold $\ge \frac{3}{d}$ is $\Dcyo = D_0$ is smooth.

\noindent\textbf{Remark.} While this paper focuses on the cases of curves, there is a generalization of Hacking's work to higher dimensional pairs $(\PP^n, D)$ in \cite{DeV19}.  The general approach above applies in this situation, and Hacking's Calabi-Yau limits are log terminal degenerations $X_0$ of $\PP^n$ containing ample divisors $D_0$.  If $D_0$ is an ample smooth Cartier divisor, $X_0$ necessarily has isolated singularities.  As in the proof of Theorem \ref{thmA:nonplanarlimits} where log terminal degenerations of $\PP^2$ with isolated singularities are used to construct non-planar limits of prime degree $p > 2$ curves,  it is of interest to construct log terminal $\QQ$-Gorenstein degenerations $X_0$ of $\PP^n$ with isolated singularities.  Certainly, a cone over a Fano hypersurface of degree $b$ in $\PP^n$ is such an example, but for a Cartier divisor $D_0$ on $X_0$, projection away from the vertex realizes $D_0$ as a degree $a$ cover of a degree $b$ hypersurface.  If $a > 1$, this is the limit of a family of degree $ab$ hypersurfaces, which is not prime, and if $a = 1$, $D_0$ is isomorphic to a degree $b$ hypersurface.  In particular, cones are not examples that can contain smooth non-hypersurface limits of prime degree hypersurfaces.  Therefore, we pose the following question, as any answers would give potential candidates for constructing non-hypersurface limits of prime degree hypersurfaces: 

\noindent\textbf{Question.} \textit{Aside from cones, what are the log terminal $\QQ$-Gorenstein degenerations of $\PP^n$ with isolated singularities?}

At least one example is known and interesting in the prime degree case: by work of Horikawa \cite{Horikawa}, there are smooth limits of quintic surfaces in $\PP^3$ that do not embed in $\PP^3$, and by \cite[Ex. 5.2]{DeV19}, these limits embed in a log terminal degeneration of $\PP^3$ with isolated singularities.  

\noindent\textit{Acknowledgements.} We would like to thank Nathan Chen, Lawrence Ein, Paul Hacking, Elham Izadi, David Jensen, J\'anos Koll\'ar, Robert Lazarsfeld, Yuchen Liu, Mirko Mauri, James M\textsuperscript{c}Kernan, Takumi Murayama, Alex Perry, Stefan Schreieder, and Burt Totaro for their insights and helpful conversations. The second author was partially supported by NSF grant DMS-1952399.

\section{Markov numbers and degenerations of the projective plane}\label{sec:degensofP2}

The goal of this section is to prove Theorem \ref{thmA:nonplanarlimits}, i.e. to show that for any Markov number $d$ there is a smooth limit of degree $d$ curves that is not planar. These limits are constructed as Cartier divisors in degenerations of the plane with isolated singularities. In the process we study some of the basic properties of the log terminal $\QQ$-Gorenstein degenerations of $\PP^2$; e.g. we compute their Class groups and show that any ample line bundle is globally generated. To show these divisors are not planar we give an upper bound on their gonality that is less than $d-1$, the gonality of a smooth degree $d$ plane curve.

We start by considering weighted projective planes that are limits of $\PP^2$. To start we show how the Markov equation arises when considering these weighted projective spaces. Suppose there is a flat family
\[
X \ra T
\]
over a pointed curve $0\in T$ such that $X$ is $\QQ$-Gorenstein and for $t\in T$ general $X_t\cong \PP^2$ and $X_0\cong \PP(p,q,r)$ is a weighted projective plane (with $p$, $q$, and $r$ coprime). Then in fact, $(p,q,r)=(a^2,b^2,c^2)$ and $(a,b,c)$ satisfy:
\begin{flalign*}
\text{(the Markov equation)}&&a^2+b^2+c^2=3abc.&&\phantom{\text{(Markov Equation)}}
\end{flalign*}
All solutions to the Markov equation are obtained by successively permuting or performing the \textit{mutation} $(a,b,c) \mapsto (a,b,3ab - c)$ starting from the minimal solution $(1,1,1)$ \cite{Markov, Mutations}. The first few triples in the Markov tree are
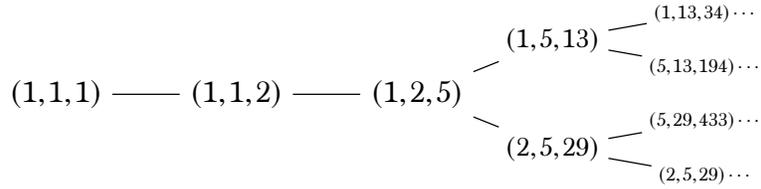
\begin{center}
\scalebox{1.2}{\begin{tikzpicture}[branch1/.style ={scale=.4}]
\node[scale=.8] at (0,0) (1){$(1,1,1)$};
\node[scale=.8] at (2,0) (2){$(1,1,2)$};
\node[scale=.8] at (4,0) (3){$(1,2,5)$};
\node[scale=.7] at (5.5,.6) (31){$(1,5,13)$};
\node[scale=.5] at (7.2,.9) (311){$(1,13,34)\cdots$};
\node[scale=.5] at (7.2,.3) (312){$(5,13,194)\cdots$};
\node[scale=.7] at (5.5,-.6) (32){$(2,5,29)$};
\node[scale=.5] at (7.2,-.3) (321){$(5,29,433)\cdots$};
\node[scale=.5] at (7.2,-.9) (322){$(2,5,29)\cdots$};
\draw (1)--(2);
\draw (2)--(3);
\draw (3)--(31);
\draw (31)--(311);
\draw (31)--(312);
\draw (3)--(32);
\draw (32)--(321);
\draw (32)--(322);
\end{tikzpicture}}
\end{center}
corresponding to the weighted projective spaces $\PP^2$, $\PP(1,1,4)$, $\PP(1,4,25)$, $\dots$. 

The fact that Markov numbers show up when considering $\QQ$-Gorenstein degenerations of $\PP^2$ is a consequence of the constancy of the anticanonical volume $(-K_{X_t})^2$ (Koll\'ar addresses the top self-intersection of a $\QQ$-Cartier divisor $D$ over $T$ in much more generality in \cite[Thm. 11]{KollarConstantVolume}).  Setting the anticanonical volume of $\PP(p,q,r)$ equal to $(-K_{\PP^2})^2 = 9$ gives
\[
(p+q+r)^2/pqr = 9.
\]
So $(p+q+r) = 3\sqrt{pqr}$. As $p$, $q$, and $r$ are all coprime, the only possibility is they are all perfect squares. In fact, one can show that for any Markov triple $(a,b,c)$ the singularities of the weighted projective space $\PP(a^2,b^2,c^2)$ can be independently smoothed in a $\QQ$-Gorenstein family which implies that every such weighted projective space is a limit of $\PP^2$ \cite[Cor. 1.2]{HackProk}. These surfaces were first studied by Manetti in \cite{Manetti} and are called Manetti surfaces.  

\begin{definition}[Manetti surfaces]
Fix a Markov triple $(a,b,c)$. Define
\[
M(a,b,c):= \PP(a^2,b^2,c^2).
\]
Define $M(b,c)$ to be the partial smoothing of the index $a^2$ singularity in $M(a,b,c)$ -- i.e. $M(b,c)$ has 2 singularities of index $b^2$ and $c^2$. Likewise define $M(c)$ to be the smoothing of the index $a^2$ and $b^2$ singularities in $M(a,b,c)$.
\end{definition}

\begin{remark}
    The singularities appearing on Manetti surfaces are examples of $T$ singularities.  In the notation of Hacking and Prokhorov, they are $T_1$ singularities, and the versal $\QQ$-Gorenstein deformation space of such a singularity is one-dimensional (c.f. \cite[Page 4]{HackProk}), so any non-trivial deformation of such a singular point smooths it completely.  Throughout this section, we will therefore use the terminology \textit{partial smoothing} of a Manetti surface $M_0$ to mean a one parameter flat family $M$ over a smooth pointed curve $0\in T$ such that $M$ has $\QQ$-Gorenstein singularities, $M_0$ is the fiber over $0 \in T$, and for any singular point in $M_0$ the local deformation of the singularity is either a smoothing or a trivial deformation.
\end{remark}

\begin{example}\label{ex:partialsmoothingconstruction}
The mutation process in the Markov tree describes how to ``connect'' two weighted projective degenerations of $\PP^2$ in a family and how to obtain the various partial smoothings.  Let $(a,b,c)$ be a Markov triple and let $(a,b,c' = 3ab-c)$ be the Markov triple obtained by a mutation.  Consider the degree $c$ embedding of the weighted projective space $M(a,b,c) = \PP(a^2,b^2,c^2) \hookrightarrow \PP(a^2,b^2,c',c)$ given by $[x_0:x_1:x_2] \mapsto [x_0^c, x_1^c, x_0x_1, x_2 ]$.  Let $[y_0:y_1:y_2:y_3]$ be the weighted coordinates on $\PP(a^2,b^2,c',c)$.  The image of $M(a,b,c)$ is given by $y_0y_1 = y_2^c$.  Similarly, the degree $c'$ embedding of $M(a,b,c')$ can be given by $y_0y_1 = y_3^{c'}$ in the same weighted projective space.

Consider the family \[ \Mc_{s,t} = (y_0y_1 = sy_2^c + ty_3^{c'}) \subset \PP(a^2,b^2,c',c) \times \AA^2_{s,t}.\]

\noindent When $s \ne 0$, $\Mc_{s,0} \cong M(a,b,c)$, and when $t \ne 0$, $\Mc_{0,t} \cong M(a,b,c')$ by construction.  When $st \ne 0$, $\Mc_{s,t} \cong M(a,b)$: as the partial derivatives of the defining equation do not simultaneously vanish, $\Mc_{s,t}$ is quasi-smooth and therefore can only acquire singularities at the four isolated singular points of the ambient space $\PP(a^2,b^2,c',c)$.  The equation avoids the singularities of index $c'$ and $c$, so can have at most two singularities at the points of index $a^2$ and $b^2$.  In a neighborhood of $[1:0:0:0]$, the surface $\Mc_{s,t}$ is defined by $y_1 = sy_2^c + ty_3^{c'}$, so has a singularity of the form $\frac{1}{a^2}(c',c) \cong \frac{1}{a^2}(b^2,c^2)$.  Similarly, at the point $[0:1:0:0]$, the surface $\Mc_{s,t}$ has a singularity of the $\frac{1}{b^2}(a^2,c^2)$.  Therefore, for $st \ne 0$, $\Mc_{s,t} \cong M(a,b)$ as claimed. 
\end{example}

In fact, Hacking and Prokhorov prove that \textit{every} log terminal $\QQ$-Gorenstein degeneration of $\PP^2$ is a Manetti surface.

\begin{theorem}[{\cite[Cor. 1.2]{HackProk}}]\label{thm:manettisurfaces}
The log-terminal $\QQ$-Gorenstein degenerations of $\PP^2$ are precisely the Manetti surfaces.
\end{theorem}

Let $M$ be a partial smoothing of a Manetti surface $M_0$ over a smooth pointed curve.  If $M_t$ is a general fiber of $M$, we say a divisorial sheaf $\Oc_{M_0}(D_0)$ on $M_0$ \textit{extends to $M_t$} if there is a divisorial sheaf $\Oc_{M}(D')$ on $M$ such that $\Oc_{M}(D')|_{M_0} \cong \Oc_{M_0}(D_0)$. We call $\Oc_{M}(D')|_{M_t}$ the \textit{extension} of $\Oc_{M_0}(D_0)$ to $M_t$. The previous Lemma shows that any divisorial sheaf on $M_0$ that is Cartier at the smoothed singularities extends to $M_t$. A priori, such an extension is not unique, but once we prove that any Manetti surface has class group $\ZZ$ and that the intersection numbers are preserved in partial smoothings, it follows that the extensions are unique.

\begin{lemma}\label{extendingdivisors}
Let $M_0$ be a log-terminal $\QQ$-Gorenstein degeneration of $\PP^2$, and let $M$ be a 1-parameter $\QQ$-Gorenstein partial smoothing of $M_0$. If $\Oc(D_0)$ is a divisorial sheaf on $M_0$ that is locally free at the smoothed singularities then there is a base change $T'\ra T$ such that $\Oc(D_0)$ extends to a divisorial sheaf on $M \times_T T'$.
\end{lemma}

\begin{proof}
We show that the obstruction to extending $\Oc_{M_0}(D_0)$ vanishes. Assume $\Oc(D_0)$ extends to a divisorial sheaf $\Oc(D_{0,n})$ on an infinitesimal deformation $M_n$ of $M_0$ over $\CC[t]/t^n$. The obstruction to extending $\Oc(D_{0,n})$ to a divisorial sheaf on $M_{n+1}$ over $\CC[t]/t^{n+1}$ is a class in
\[
\Ext^2_{\Oc_{M_n}}(\Oc_{M_n}(D_{0,n}),\Oc_{M_0}(D_0))
\]
(see e.g. \cite[\href{https://stacks.math.columbia.edu/tag/08L8}{Tag 0ECH}]{stacks-project}). The local-to-global spectral sequence shows that all the obstructions are local and supported at the non-smoothed points ($H^2(M_n,\mathcal{H}\mathrm{om}(\Oc(D_{0,n}),\Oc(D_0))) = H^2(M_0,\Oc_{M_0}) = 0$, and $\mathcal{E}\mathrm{xt}^1(\Oc(D_{0,n}),\Oc(D_0))$ is supported on the non-smoothed points). So the obstruction lives in
\[
H^0(M_n,\mathcal{E}\mathrm{xt}^2(\Oc(D_{0,n}),\Oc(D_0)))
\]
This shows the obstructions are local and supported at the non-smoothed point. But locally at this point, the deformation of $M_0$ is trivial, thus there is no local obstruction and $\Oc(D_{0,n})$ extends. Lastly, it may be necessary a priori to make a base change $T'\ra T$ to make the above deformation algebraic.
\end{proof}

\begin{theorem}\label{thm:classpicofmanettisurfaces} Let $M_0$ be a $\QQ$-Gorenstein degeneration of $\PP^2$.
\begin{enumerate}
\item $\Cl(M_0)=\ZZ$ and $\Pic(M_0)=\ZZ$. We write $\Oc_{M_0}(1)$ for the positive generator of $\Cl(M_0)$.
\item If $A$ is the direct sum of the local class groups at the singular points of $M_0$, then the sequence:
\[
0\ra \Pic(M_0)\ra \Cl(M_0)\ra A\ra 0
\]
is exact. In particular, the map from $\Pic(M_0)$ to $\Cl(M_0)$ is multiplication by $\alpha^2=|A|$.
\item With the notation from (2), $\Oc_{M_0}(1)$ has self intersection number $1/\alpha^2$.
\item Let $M \to T$ be a 1-parameter $\QQ$-Gorenstein partial smoothing of $M_0$. Let $B$ be the local class group of the smoothed singularities and let $\beta=\sqrt{|B|}\in \ZZ$. A divisor $D_0\in \Cl(M_0)$ extends to a divisor $D = D_{T'}$ on $M \times_T T'$ for a base change $T' \to T$ if and only if $\Oc_{M_0}(D_0) = \Oc_{M_0}(m \beta)$ for some integer $m$.  Furthermore, over a general point, the extension $D$ is in $|\Oc_{M_t}(m)|$.
\item If $D$, $D'$ are two $\QQ$-Cartier Weil divisors on $M$ then the intersection numbers $(D|_{M_t})\cdot (D'|_{M_t})$ are independent of $t\in T$.
\end{enumerate}
\end{theorem}\label{deformationtheory}

\begin{proof} 
We know that $M_0$ is a partial smoothing of a weighted projective space $\PP:=\Pabc$. For weighted projective spaces the map from the class group to the local class groups is surjective. Thus it follows from Lemma \ref{extendingdivisors} that the map
\[
\Cl(M_0)\ra A
\]
is surjective, which shows the sequence in (2) is exact.

The divisor $\Oc(-K_{\PP})=\Oc_{\PP}(3abc)$ extends to the divisor $K_{M_0}$ on $M_0$. Let $B$ (resp. $A$) be the local class group of the smoothed (resp. unsmoothed) singularities. By Lemma~\ref{extendingdivisors}, $\Oc_\PP(|B|)$ extends to $M_0$. Set $\beta = \sqrt{|B|}$, so $\beta=\mathrm{gcd}(|B|, 3abc$), and set $\alpha=\sqrt{|A|}$. Then $\Oc_\PP(\beta)$ extends to a divisorial sheaf $\Oc_{M_0}(1)$ on $M_0$, and $\Oc_{M_0}(1)$ generates the local class group of the unsmoothed singularities.  Now we wish to prove that $\Oc_{M_0}(1)$ generates $\Cl(M_0)$.

It is known that $\Pic(M_0)=\ZZ$ (\cite[Lem. 2.1, Prop. 6.3]{Hacking}). Let $\Oc_{M_0}(1)\in \Cl(M_0)$ be the divisor on $M_0$ from the previous paragraph (i.e. $\Oc_\PP(\beta)$ extends to $\Oc_{M_0}(1)$ on $M_0$). We will show that $\Oc_{M_0}(1)$ generates $\Cl(M_0)$. Consider the following diagram:
\[
\begin{tikzcd}
0\arrow[r] & \ZZ\cdot \Oc_{M_0}(\alpha^2)\arrow[r] \arrow[d]& \ZZ\cdot \Oc_{M_0}(1)\arrow[r]\arrow[d] & \ZZ/\alpha^2\ZZ\arrow[r]\arrow[d]  &0\\
0 \arrow[r] &\Pic(M_0)\arrow[r] & \Cl(M_0)\arrow[r] & A\arrow[r]  &0.
\end{tikzcd}
\]
It suffices to show that the map $\ZZ\cdot \Oc(\alpha^2)$ to $\Pic(M_0)$ is surjective.

As $K_{M_t}^2$ is constant in the family $M \to T$ and the restriction of every divisor $D|_{M_t}$ is numerically a rational multiple of $K_{M_t}$, the intersection numbers of divisors are constant in the family (proving (5)). On the central fiber, if $D_0 \in |\Oc_\PP(\beta)|$, then $D_0^2 = \frac{\beta^2}{\alpha^2\beta^2} = \frac{1}{\alpha^2}$.  Therefore, $\Oc_{M_0}(1)^2 = \frac{1}{\alpha^2}$ (proving (3)).

Now, let $\Lc$ be the ample generator of $\Pic(M_0)$.  Then $\Lc \equiv_{\mathrm{num}} \Oc_{M_0}(\mu)$ so $\Oc_{M_0}(\mu)\cdot \Oc_{M_0}(1)$ is an integer.  Let $D_0 \in |\Oc_{M_0}(1)|$.  As $\alpha^2D_0$ is Cartier, $\mu \le \alpha^2$.  Therefore, because $\Oc_{M_0}(\mu)\cdot \Oc_{M_0}(1) = \mu \frac{1}{\alpha^2}$, we must have $\mu = \alpha^2$.  This implies that $\Lc = \Oc(\alpha^2 D_0) $ and $\Oc(\alpha^2 D_0)$ generates $\Pic(M_0)$. This completes the proof of (1).

To prove (4), assume that (after possibly base changing) that $\Oc_{M_0}(\mu)$ extends to an ample generator $\Oc_{M_t}(1)$ of $\Cl(M_t)$ for a general fiber $M_t$ of $M$. Then by (5) we have $\Oc_{M_0}(\mu)^2 = \Oc_{M_t}(1)^2$. By (3) the left hand side is $\mu^2/\alpha^2$ (where $\alpha^2 = |A|$ where $A$ is the local class group of the singularities of $M_0$). And by (3) the right hand side is $\beta^2/\alpha^2$ which proves (4).
\end{proof}

\begin{theorem}\label{thm:ampleisbpfformanetti} Let $M_0$ be the central fiber of a $\QQ$-Gorenstein degeneration of $\PP^2$.
\begin{enumerate}
\item If $M$ is a 1-parameter partial smoothing of $M_0$ and $D\subset M$ is a divisor flat over $T$, then the dimension of the the cohomology $h^i(M_t,\Oc_{M_t}(D_t))$ is constant in the family.
\item If $D_0\subset M_0$ is an ample Cartier divisor, then $\Oc_{M_0}(D_0)$ is globally generated.
\end{enumerate}
\end{theorem}\label{globallygenerated}

\begin{proof}
For (1), as $M$ is a family of surfaces, by semicontinuity it suffices to show that \[h^1(M_0,\Oc_{M_0}(D_0))=0\] for any divisor $D_0$ on $M_0$ a $\QQ$-Gorenstein degeneration of $\PP^2$. For any such divisor, consider a 1-parameter $\QQ$-Gorenstein degeneration $M'$ of $M_0$ to $\PP(a^2,b^2,c^2)$ for some Markov triple $(a,b,c)$ such that there is a divisor $D'\subset M'$ flat over $T$ with $\Oc(D'_t)\cong \Oc(D_0)$ for $t$ general. Any divisor on $\PP(a^2,b^2,c^2)$ has vanishing intermediate cohomology so we are done by semicontinuity.

It suffices to prove (2) for the ample generator $D_0\in\Pic(M_0)$. We claim that it suffices to show there is a divisor $C_0\in |D_0|$ avoiding the singularities of $M_0$. First, as $D_0$ is the generator of $\Pic(M_0)$ such a divisor is reduced and irreducible (any component of $C_0$ would be Cartier as it avoids singularities). Assuming the existence of such a $C_0$ we see that the base locus of $|D_0|$ is contained in $C_0$. Consider the exact sequence on cohomology induced by the short exact sequence:
\[
0\ra \Oc_{M_0} \ra \Oc_{M_0}(D_0)\ra \Oc_{C_0}(D_0)\ra 0.
\]
By the previous paragraph $H^1(M_0,\Oc_{M_0})=0$ so the base locus of $\Oc_{M_0}(D_0)$ and $\Oc_{C_0}(D_0)$ are equal. To show $\Oc_{C_0}(D_0)$ is base point free it suffices to show for any point $p\in C_0$ with maximal ideal $\mf_p\subset \Oc_{C_0}$ that $H^1(C_0,\mf_p(D_0))=0$. By Serre Duality
\[
H^1(C_0,\mf_p(D_0))\cong \Hom(\mf_p(D_0),\omega_{C_0})^\vee.
\]
As $D_0$ is Cartier and $C_0\subset M_0$ avoids the singularities of $M_0$:
\[
\Hom(\mf_p(D_0),\omega_{C_0}) = \Hom(\mf_p,\omega_{M_0}|_{C_0}).
\]
Finally, the degree of $\omega_{M_0}|_{C_0}<-1$. As $C_0$ is reduced and irreducible $\Hom(\mf_p,\omega_{M_0}|_{C_0})=0$.

So, it suffices to show there exists a divisor $C_0\in |D_0|$ which avoids the singularities of $M_0$. Let $M' \to T'$ be a 1-parameter $\QQ$-Gorenstein degeneration of $M_0$ to a weighted projective space $\PP(a^2,b^2,c^2)$ (for some Markov triple $(a,b,c)$), so the general fiber $M'_t$ of $M'$ is $M'_t \cong M_0$, and the central fiber is $M'_0 \cong \Pabc$. Further assume $a\le b\le c$ and the Markov triple has $c$ as small as possible. Let $A$ be the local class group of the smoothed singularities of $\PP(a^2,b^2,c^2)$ and $B$ is the local class group of the unsmoothed singularities. Set $\alpha=\sqrt{|A|}$ and $\beta = \sqrt{|B|}$. Then (after possible base change) $\Oc(\alpha \beta^2)$ extends to a divisorial sheaf $D'\subset M'$ such that $\Oc_{M'_t}(D'_t)\cong \Oc_{M_0}(D_0)$ for $t$ general. As the dimension of global sections of $\Oc_{M'_t}(D'_t)$ are constant in this family, the base locus $Z\subset M'$ of $|D'_t|$ is closed. So it suffices to show there is a divisor $C'_0\in \PP(a^2,b^2,c^2)$ that avoids the unsmoothed singularities.

As $\beta^2=|B|$ divides $\alpha\beta^2$, it is easy to see that there are divisors in $|\Oc(\alpha \beta^2)|$ that avoid the unsmoothed singularities. Specifically, if $\beta=c$, $\alpha=ab$, and $\PP(a^2,b^2,c^2)$ has coordinates $x,y,z$ then $(z^{ab}=0)$ avoids the unsmoothed singularities. Or if $\beta = bc$ and $\alpha = a$ then the divisor $(z^{ab^2}+y^{ac^2}=0)$ avoids the unsmoothed singularities.
\end{proof}

Now we have developed the machinery needed to prove Theorem \ref{thmA:nonplanarlimits}.

\begin{definition}
We say a prime number $p$ is a \textit{Markov prime} if it appears in a Markov triple.
\end{definition}

\begin{theorem}\label{smoothlimits1}
Let $(a,b,c)\ne (1,1,1)$ be a Markov triple in non-decreasing order. If $d>2$ is a multiple of $c$ then there is a smooth limit of degree $d$ plane curves that is not planar. In particular, if $p>2$ is a Markov number that is prime, there is a nonplanar degeneration of degree $p$ plane curves.
\end{theorem}

\begin{remark}
This proves Theorem~\ref{thmA:nonplanarlimits}.
\end{remark}

\begin{proof}
First, we construct a smooth Cartier divisor on $M(c)$ of the appropriate degree, and show that it extends to a smooth divisor in the degeneration of $\PP^2$ to $M(c)$. In the special case $c=2$ (so $d$ is even, and at least 4) we can degenerate a family of even degree curves to a smooth double of a degree $d/2$ curve. Projecting from a point on the curve shows the gonality of such a curve is at most $d-2$. But the gonality of a degree $d$ plane curve is $d-1$, so we see that this double cover is not planar. So we proceed by assuming $c>2$, and we similarly show the gonality is too small.

By Theorem \ref{thm:classpicofmanettisurfaces}(3), in the specialization of $\PP^2$ to $M(c)$ the line bundle $\Oc_{\PP^2}(nc)$ specializes to the line bundle $\Oc_{M(c)}(nc^2)$. By Theorem \ref{thm:ampleisbpfformanetti}(2), $\Oc_{M(c)}(nc^2)$ is globally generated, so there is a smooth divisor $C\in |\Oc_{M(c)}(nc^2)|$. By Theorem \ref{thm:ampleisbpfformanetti}(1), this is a specialization of a family of curves of degree $nc$.

Now we want to show that $C$ is not a plane curve. If it were planar, it would have to have degree $nc$ and gonality $nc-1$. The idea is to show that $C$ admits a low degree pencil and therefore smaller gonality. Consider the divisorial sheaf $\Oc_{M(c)}(ab)$ on $M(c)$. Specializing $M(c)$ to $M(a,b,c)$ gives a specialization of $\Oc_{M(c)}(ab)$ to the divisorial sheaf $\Oc_{M(a,b,c)}(a^2b^2)$ on $M(a,b,c)=\Pabc$. This has two natural sections $x^{b^2}$ and $y^{a^2}$, so by Theorem \ref{thm:ampleisbpfformanetti}(1) we see that $h^0(\Oc_{M(c)}(ab))\ge 2.$ This is the desired pencil on $M(c)$. Restricting to $C$ shows that the gonality of $C$ is at most $nab$. Now we claim that $ab< c-1$ and thus the gonality of $C$ is too small.

We prove the claim that $ab < c-1$ by induction.  The first Markov triple with $c > 2$ is $(1,2,5)$ which clearly satisfies $ab < c -1$.  For any Markov triple, there is a finite sequence of mutations connecting $(1,2,5)$ to $(a,b,c)$ such that the sum of the elements in the Markov triple increases at each step \cite[Prop. 3.7(a)]{Mutations}. So, assume that $(a,b,c)$ is a triple with $a \le b \le c$ and the inductive hypothesis holds for every triple in the mutation sequence from $(1,2,5)$ to $(a,b,c)$, with $c > 2$.  Consider the previous triple $(a,c', b)$ (or $(c',a,b)$) where $c' = 3ab -c$.  Note that we must have $c' \le b$ by hypothesis that the sum increases at each stage as shown in \cite[Prop. 3.7(b)]{Mutations}.  Then, the inductive hypothesis says $ac' < b -1$, so $a(3ab-c) < b-1$.  Rearranging, we see that $3a^2b - b < ac - 1$, and because $2a^2b - b> 0$, this implies that $a^2b \le ac -1$, so $ab < c - \frac{1}{a} < c - 1$.  Therefore, $ab < c-1$.
\end{proof}

\begin{example}\label{ex:quintics}
The first non-trivial example is the case $p=5$. In this case, Griffin proved (\cite[Thm. 1.A]{Griffin}) that any smooth hyperelliptic genus 6 curve embeds in $\PP(1,2,13)$ as \[ z^2 = g_{26}(x,y)\] in the weighted coordinates $[x:y:z]$ on $\PP(1,2,13)$, where $g_{26}(x,y)$ is a degree 26 polynomial such that $g_{26}(0,y) \ne 0$ and the polynomial $h_{13}(u,y)$ has 13 distinct roots, where $h_{13}$ is obtained from $g_{26}$ by $u = x^2$.  The polynomial $g_{26}$ is determined by choice of Weierstrass point on the curve.

These curves further admit an embedding as Cartier divisors on $M(5)$.  The surface $M(5)$ can be explicitly realized as a hypersurface of degree 26 in $\PP(1,2,13,25)$ (c.f. Example \ref{ex:partialsmoothingconstruction}): in terms of weighted coordinates $[x:y:z:w]$, $M(5)$ can be given by \[ (xw = z^2 - g_{26}(x,y)) \subset \PP(1,2,13,25) \] where $g_{26}$ is any degree 26 polynomial such that $g_{26}(0,y) \ne 0$.  To compare with the notation of Example \ref{ex:partialsmoothingconstruction}, make the change of coordinates $w \mapsto w-\frac{1}{x}(g_{26}(x,y) - g_{26}(0,y))$ and scaling of $y$ so this can be expressed as $xw = z^2 + y^{13}$.  

Therefore, if $M(5) = (xw = z^2 - g_{26}(x,y))$ in $\PP(1,2,13,25)$, any hyperelliptic curve can be written as the Cartier divisor $(w = 0)$ on $M(5)$.  By Theorem \ref{thm:classpicofmanettisurfaces} this Cartier divisor extends to a Cartier divisor on a smoothing of $M(5)$ to $\PP^2$, so every hyperelliptic curve of genus 6 can be realized as a limit of quintic plane curves in this way.    
\end{example}

\begin{remark}\label{rem:Zagier}
In \cite{Zagierdensity}, Zagier proved that the Markov counting function grows asymptotically like $\log(n)^2$. As the prime counting function grows like $n/\log(n)$ the density of primes that are Markov numbers is 0. Therefore, Conjecture~\ref{conjC} implies that for almost all prime numbers $p$, every smooth limit of a degree $p$ plane curve is a plane curve.
\end{remark}

\section{Hacking's Calabi-Yau limits}\label{sec:CYlimits}

Given a family of plane curves $C_t \subset \PP^2$ over a punctured base $\AA^1 \setminus \{ 0 \}$, one may ask how to complete the family, or if there is a unique way to do so.  A useful framework for this problem is to study not only degenerations of the curve $C_t$, but rather the pair $(\PP^2, C_t)$.  For curves of degree $d \ge 4$, to study the limits of these pairs, we introduce the moduli space of stable pairs.  Although we will rely mostly on Hacking's work in \cite{Hacking}, we include several relevant definitions and results here, and direct the interested reader to \cite{KollarModuliBook}.

\begin{definition}[KSBA stability; see Definition 8.6, 8.7 in \cite{KollarModuliBook}] 
The moduli space of stable pairs parametrizes semi-log canonical pairs $(X,D)$, with $X$ projective, of fixed volume and dimension, such that $K_X + D$ is ample.
\end{definition}

Defining a family of such pairs requires great care and subtlety, see \cite[Ch. 7, 8]{KollarModuliBook}.

\begin{theorem}[Theorem 8.9(2) in \cite{KollarModuliBook}]\label{stablelimit}
Up to base change, any family of stable pairs
\[
(X, cD) \to T^\times
\]
over a punctured one-dimensional base $T^\times$ can be completed uniquely to a stable family $(\overline{X}, c \overline{D}) \to T$, such that $K_{\overline{X}/T} + c \overline{D}$ is relatively ample and each fiber $(\overline{X}_t, c \overline{D}_t)$ has slc singularities.
\end{theorem}

In particular, given any family of degree $d \ge 4$ plane curves $C$ over a punctured one-dimensional base $T^\times$, for $c > \frac{3}{d}$, by regarding the family as a family of pairs $(\PP^2 \times T^\times, cC)$, we obtain a unique stable limit $(X_0,cD_0)$ of the family.  The surface $X_0$ is a slc degeneration of $\PP^2$ and $D_0$ is a limit of the family of plane curves. 

\begin{example}
Although the previous theorem guarantees a unique limit for every $c$, these vary with $c$ in a precise way, and as $c$ varies, there is a wall-crossing phenomenon \cite{KSBwallcrossing}.  For example, take a family of smooth quartic curves $C_t \subset \PP^2$ degenerating to a curve $C_0 \subset \PP^2$ with a single cusp, a singularity of the form $x^2 + y^3 = 0$.  This pair $(\PP^2, C_0)$ is the stable limit of the family of smooth curves for $\frac{3}{4} < c < \frac{5}{6}$.  However, for $c > \frac{5}{6}$, the stable limit is a pair $(X,C_0')$, where $X$ is a non-normal surface $X_1 \cup X_2$, where each $X_1$ has cyclic quotient singularities along the double locus, and $C_0'$ is a reducible curve $C_0' = C_1 \cup E$ (one component in each component of $X$), where $E$ is an elliptic tail meeting $C_1$ at a single point on the double locus of $X$.  This example (and related wall crossings) were studied in \cite{Hassett}. 
\end{example}

\begin{definition}\label{coefficient1}
Let $D$ be a smooth family of curves over a base $T$ with special fiber $0 \in T$ such that $D_t$ is a plane curve of degree $d > 3$ for $t \ne 0$.  Let $T^\times = T \setminus \{0\}$ and $D^\times = D \times_T T^{\times}$.  Consider the family $(\PP^2 \times T^\times, D^\times)$.  By Theorem \ref{stablelimit}, there is a unique limit of the family such that the central fiber $(X_0, D'_0)$ is slc and $K_{X_0} +D'_0$ is ample.  By adjunction, this implies that $D'_0$ is nodal and $K_{D'_0}$ is ample, hence by applying Theorem \ref{stablelimit} to the family of curves $D^\times$, $D'_0$ is the unique stable limit of this family.  Therefore, $D'_0 = D_0$, and we will call $(X_0, D_0)$ the coefficient 1 limit of $(\PP^2 \times T^\times, D^\times)$ and we will call the family $(X,D)$ the coefficient 1 family.  
\end{definition}

In general, determining the possible degenerations of $\PP^2$ that appear in these limits is quite difficult.  As a first indication, it is not even clear how many components these degenerations should have.  However, in \cite{Hacking}, the author studies stable pairs $(X, cD)$ where $c = \frac{3}{d} + \epsilon$ is as small as possible.  In fact, Hacking is able to obtain a moduli space of \textit{semistable} pairs \cite[Defn. 2.4]{Hacking} $(X,\frac{3}{d}D)$ satisfying:     
    \begin{itemize}
        \item The surface $X$ is normal and log terminal, 
        \item The pair $(X, \frac{3}{d} D)$ is log canonical, and $dK_X + 3D \sim 0$,
        \item $X$ admits a $\QQ$-Gorenstein smoothing to $\PP^2$.  
    \end{itemize}

Furthermore, Hacking proves that any family of pairs $(\PP^2, C_t)$ can be completed to a family of semistable pairs.  This completion is not necessarily unique, but its existence is sufficient for this paper.  

\begin{definition}
Given a family $(\PP^2 \times T^\times, D^\times)$ of smooth degree $d$ plane curves over $T^\times$, we will call any such family of semistable pairs \textit{Hacking's Calabi-Yau family}, denoted by $(\Xcy, \Dcy)$.  
\end{definition}

A main benefit to considering these semistable pairs is that there exists an explicit (infinite) list of log terminal $\QQ$-Gorenstein degenerations of $\PP^2$, enumerated by the Markov numbers.  However, the curves appearing on these surfaces may be quite singular.  In particular, the log canonical threshold can be as small as $\frac{3}{d}$.  The trade-off between simplifying the surface and complicating the curve will be explored in the following sections.  

\begin{remark}
An alternative perspective to constructing Hacking's Calabi-Yau limits comes from K-moduli.  For $\epsilon \in (0, \frac{3}{d})$, there is a K-moduli space parameterizing pairs $(X, (\frac{3}{d}-\epsilon)D)$ arising as degenerations of plane curves of degree $d$ (c.f. \cite{ADL19}).  These depend on $\epsilon$, but for for $0< \epsilon \ll 1$, the K-moduli spaces are isomorphic (\cite[Thm 1.2]{ADL19}), i.e. for any pair $(X, (\frac{3}{d}  - \epsilon)D)$ appearing in the K-moduli space is K-semistable for all $0< \epsilon \ll 1$.  Because any K-semistable pair is klt, this implies that the pair $(X, \frac{3}{d}D)$ is a semistable limit.  So at least one of Hacking's Calabi-Yau limits of a family of plane curves is reproduced from K-moduli.
\end{remark}

For any fixed degree $d$, it is straightforward to write down the possible surfaces $\Xcyo$ that may appear using the classification from Theorem \ref{thm:manettisurfaces}, the index bound from \cite[Thm. 4.5]{Hacking}, and a log canonical threshold computation.  In particular, the following propositions give all possible surfaces $\Xcyo$ containing limits of degree $5$ or degree $7$ curves. 

\begin{proposition}\label{degree5surfaces}
Let $D$ be a projective family of curves over a smooth curve pointed curve $0 \in T$ such that the general fiber is a smooth plane curve of degree 5. If $(\Xcyo,\Dcyo)$ is Hacking's Calabi-Yau limit over $0$ then $\Xcyo$ is either $\PP^2$, $\PP(1,1,4)$, $M(5)$, or $\PP(1,4,25)$.
\end{proposition}

\begin{proof}
By \cite[Thm. 4.5]{Hacking}, the index of the canonical divisor of $\Xcyo$ is at most $5$ at each point, and the index of $K_M$ of any surface $M = M(a,b,c)$, $M(a,b)$, or $M(a)$ is given by $a,b$, or $c$, so we may assume that $a,b,c \le 5$.  The only surfaces with this property are $\PP^2$, $M(2) = \PP(1,1,4)$, $M(5)$, or $M(2,5) = \PP(1,4,25)$.
\end{proof}

\begin{proposition}\label{degree7noindex5surface}
Let $D$ be a projective family of curves over a smooth curve pointed curve $0 \in T$ such that the general fiber is a smooth plane curve of degree 7. If $(\Xcyo,\Dcyo)$ is Hacking's Calabi-Yau limit over $0$ then $\Xcyo$ is either $\PP^2$ or $\PP(1,1,4).$
\end{proposition}

\begin{proof}
$\Xcyo$ is a surface $M$ that is $M(a,b,c)$, $M(a,b)$, or $M(a)$ for some unordered Markov triple $(a,b,c)$. By \cite[Thm. 4.5]{Hacking}, the index of the canonical divisor of $M$ is at most $7$ at each point, and the index of $K_M$ at each singular point is given by $a$, $b$, or $c$, so we can assume $a,b,c\le 7$. So the only Markov triples to consider are $(1,1,1)$, $(1,1,2)$, and $(1,2,5)$. So we only need to consider the triple $(1,2,5)$ and rule out the surfaces $M(2,5)$ and $M(5)$.  However, limits of degree 7 curves on $M(5)$ or $M(2,5)$ have log canonical threshold at most $\frac{1}{4}$ (see e.g. \cite[Prop. 9.13]{ADL19}), which is less than $\frac{3}{7}$, so cannot appear as Hacking's Calabi-Yau limits. 
\end{proof}

\section{General results on normal families of curves and normalizing S2 varieties}\label{sec:Dcynorm}

In \S\ref{sec:CYlimits} we started with a smooth projective family of curves
\[
D\ra T
\]
over a 1-dimensional base $T$. Assuming that a general member is a plane curve, we showed there is a natural, one-parameter family $\Dcy$ such that every fiber is planar or lives in a degeneration of $\PP^2$. One difficulty that arises is that the surface $\Dcy$ need not be normal. On the other hand, $D$ is a smooth family (and therefore normal). The purpose of this section is to understand the basic relationship between $\Dcy$, its normalization $\Dnorm$, and $D$. These results are used extensively in the next section to study the combinatorics of the central fiber of Hacking's Calabi-Yau family, especially when $\Dcy$ is not a normal surface.

\subsection{Contracting to a normal surface}
Consider a commutative diagram of varieties:
\[
\begin{tikzcd}
S'\arrow[r,"\mu"]\arrow[dr]&S\arrow[d]\\
&T
\end{tikzcd}
\]
such that $\mu$ is a birational map of normal surfaces, $T$ is a curve, the maps to $T$ are projective and flat with general fiber a smooth and irreducible curve of genus $g$. Assume there is a marked point $0\in T$ and that the fiber $S'_0$ is a reduced and normal crossing curve.

\begin{definition}\label{defn:delta}
Recall that for $C$ a reduced curve with normalization $f:\Ct\ra C$ and $p\in C$ is a point, the \textit{$\delta$-invariant} of $C$ at $p$ is the number:
\[
\delta_p=\length_p\left(f_*(\Oc_{\Ct})/\Oc_C\right).
\]
\end{definition}

The sum of the $\delta$-invariants controls the difference between the arithmetic and geometric genuses of $C$:
\[
\chi(\Oc_{\Ct})-\chi(\Oc_C) = \sum_{p\in C} \delta_p.
\]

Let $E=\mu^{-1}(p)\subset S'_0$ be an exceptional divisor (which is necessarily reduced, but possibly reducible). The $\delta$-invariant of the point $p\in S_0$ can be computed as follows.

\begin{lemma}\label{deltabranches}
With the assumptions above:
\begin{enumerate}
\item The $\delta$-invariant at $p\in S_0$ can be computed by the following formula:
\[
\delta_p=(\#\text{ of branches at }p\in S_0)-\chi(\Oc_E).
\]
\item If there is a component $C\in S'_0$ such that $p_g(C)=p_a(S'_0)=g$ then either $C$ maps isomorphically onto its image in $S_0$ or $C$ is contracted to a point $q\in S_0$ and for all $p\in S_0:$
\[
\delta_p(S_0) = \left\{ \begin{array}{ll}(\#\text{ of branches at }p\in S_0)-1&\text{ if $p\ne q$}\\ (\#\text{ of branches at }p\in S_0)-1+g&\text{ if $p=q$.}\end{array}\right.
\]

\end{enumerate}
\end{lemma}

\begin{proof}
For (1), by contracting all exceptional divisors of $\mu$ that don't map to $p$ we can assume that $\mu\cl S'\ra S$ is an isomorphism away from $p$. So $\mu$ is an isomorphism away from $p$. Let $C'=\mu^{-1}(S_0)$ be the strict transform of $S_0$ in $S'$. Thus $S'_0 = C'\cup E$, and locally at $p$, $C'$ normalizes $S_0$.

Consider the exact sequence of sheaves on $S'$:
\[
0\ra I \ra \Oc_{S'_0} \ra \Oc_{C'}\ra 0.
\]
As $S'\ra T$ is flat, $\chi(\Oc_{S'_0})=1-g=\chi(\Oc_{S_0})$. Thus $\delta_p = -\chi(I).$

The ideal $I$ is supported on the exceptional divisor $E$ and $S'_0$ is nodal in a neighborhood of $E$; thus $I\cong \Oc_E(-E\cap C')$ (where the intersection $E\cap C'$ is considered as a reduced divisor on $E$). As $C'$ is nodal near $E$, $\Oc_E(-E\cap C')$ is a line bundle with degree:
\[
\deg(\Oc_E(-E\cap C')) = -(\#\text{ of branches at }p\in S_0).
\]
So by Riemann-Roch:
\[
\delta_p = -\chi(I) = -\chi(\Oc_E(-E\cap C')) = (\#\text{ of branches at }p\in S_0) - \chi(\Oc_E).
\]

For (2), note that the dual graph of $S'_0$ is necessarily a tree, and the dual graph of any connected curve $E\subset S'_0$ is either (a) a tree of rational curves, in which case $\chi(\Oc_E) = 1$, or (b) a tree that contains $C$, in which case $\chi(\Oc_E) = 1-g$. The result then follows from (1).
\end{proof}

\subsection{Normalizing S2 varieties}

The goal of this section is to prove a general result about S2 varieties, which we believe is well known to experts. We show that all complex S2 varieties can be constructed by \textit{codimension one gluing conditions} on their normalizations.  

\begin{remark}\label{DcyisS2}
The main application of this section will be to studying birational models of $\Dcy$.  As $\Xcy$ is log terminal and $\QQ$-factorial (\cite[Lem. 2.11]{Hacking}), Corollary 5.25 in \cite{KollarMori} implies $\Oc_\Dcy$ is CM and hence $\Dcy$ is S2.  
\end{remark}

\begin{definition}
Let $Y\subset Z$ be a subscheme of a normal pure-dimensional scheme $Z$. Define the \textit{codimension one part of $Y$}, denoted $Y_\tdiv \subset Y$ to be the maximal subscheme of $Y$ which has pure codimension 1 in $Z$.
\end{definition}

Let $X$ be a complex S2 variety and let
\[
\nu \cl X^\nu \ra X
\]
be the normalization map. Let $F\subset X$ be a divisor.

\noindent \textit{Assume that -- at least set-theoretically -- $F$ contains the non-normal locus of $X$.}

\noindent Suppose
\[
F=F_1\cup \cdots \cup F_\ell
\]
is a union of (possibly reducible) divisors. Set
\[
D_i:= \nu^{-1}(F_i)_\tdiv.
\]
For any $k>0$ as $X$ is normal, there is a well-defined subscheme $kD_i\subset X$ (resp. $kD$) and it has a reflexive ideal sheaf $\Oc_{X^{\nu}}(-kD_i)$. We define $kF_i$ (resp. $kF$) to be the scheme theoretic image of $kD_i$ (resp. $kD$).

\begin{theorem}\label{thm:intersectionsinnormalization}
In the set-up above, assume that
\[
D=D_1\sqcup\cdots\sqcup D_\ell
\]
is the disjoint union of the $D_i$. For every $k>0$ the pushout
\[
\begin{tikzcd}
kD \arrow[r]\arrow[d]& (kF_1\sqcup \cdots \sqcup kF_\ell)\arrow[d]\\
X^\nu \arrow[r]& X^\nu \sqcup_{kD} (kF_1\sqcup \cdots \sqcup kF_\ell)=:X_k
\end{tikzcd}
\]
exists, there are maps $\nu_k\cl X_k \ra X$, and for $k$ sufficiently large $\nu_k$ is an isomorphism.
\end{theorem}

\begin{remark}
In other words, the S2 variety $X$ can be constructed via divisorial gluing conditions. Set-theoretically, this says that if divisors intersect in $X$ then the codimension 1 parts of their preimages intersect in $X^\nu$. 
\end{remark}

More precisely:

\begin{corollary}\label{cor:pointaboveintpoint}
Suppose that $F=F_1\cup F_2$ and there is a point $x\in X$ in the intersection $F_1\cap F_2$. Then the divisors $D_1$ and $D_2$ intersect over $x$, i.e. $x\in \nu(D_1\cap D_2)$.
\end{corollary}

Similarly there is a result for the analytic branches of $F$ at a point:

\begin{corollary}\label{cor:branchesmustmeetinnorm}
Suppose that $x\in F\subset X$ and $I=\{F_1,\cdots,F_\ell\}$ is the set of analytic branches of $F$ at $x$. There is no partition:
\[
I = I_1\sqcup I_2
\]
(with $I_1$ and $I_2$ nonempty) such that the codimension one preimages of the branches in $I_1$ do not meet the codimension 1 preimages of the branches in $I_2$.
\end{corollary}

The proof of Theorem \ref{thm:intersectionsinnormalization} uses the following lemma.

\begin{lemma}\label{lem:isoincodim2}
Let $X$ and $Y$ be complex varieties and assume $X$ is S2. If 
\[
f\cl Y\ra X
\]
is a finite, proper, birational map that is an isomorphism away from codimension 2, then $f$ is an isomorphism.
\end{lemma}

\begin{proof}[Proof of Lemma]
It suffices to check $f$ is an isomorphism locally. Assume that $X=\Spec(A)$ and $Y=\Spec(B)$. Consider the sequence of $A$-modules:
\[
0\ra A \ra B \ra \Qc \ra 0.
\]
Then $f$ is an isomorphism if and only if $\Qc=0$. We prove that $\Qc$ has no associated points, so $\Qc=0$. Note, that any associated point of $\Qc$ has codimension at least 2 by assumption. Let $\pf$ be any prime ideal in $A$ of codimension at least $2$. We have
\[
\Hom_{A}(A/\pf,B)\ra \Hom_{A}(A/\pf,\Qc) \ra \Ext^1_{A}(A/\pf,A).
\]
The first module vanishes because $B$ is torsion free as an $A$-module. The last module vanishes because $A$ is S2. So $\Hom_A(A/\pf,\Qc)=0$ for all $\pf$ prime of codimension at least 2. Thus $\Qc=0$, so $f$ is an isomorphism.
\end{proof}

\begin{proof}[Proof of Theorem \ref{thm:intersectionsinnormalization}]
The existence of the pushout is guaranteed by Ferrand's work \cite[Thm. 7.1]{Ferrand} (alternatively, see \cite[\href{https://stacks.math.columbia.edu/tag/0ECH}{Tag 0ECH}]{stacks-project}). Given the existence of $\nu_k$ it suffices to check it is an isomorphism over an affine open set for $k$ sufficiently large. Set:
\begin{itemize}
\item $X=\Spec(A)$, $X_k = \Spec(A_k)$, $X^\nu = \Spec(B)$,
\item $(kD_1\sqcup \cdots \sqcup kD_\ell) = \Spec((R_k)_1\times \cdots \times (R_k)_\ell)=\Spec(R_k)$,
\item $(kF_1\sqcup \cdots \sqcup kF_\ell) = \Spec((S_k)_1\times \cdots \times (S_k)_\ell)=\Spec(S_k)$.
\end{itemize}
The construction of the pushouts says $A_k$ is a fiber product of rings:
\[
A_k= B\times_{R_k} S_k.
\]
Each map $(S_k)_i\ra (R_k)_i$ is injective and each map $B\ra (R_k)_i$ is a surjection. The goal is to show that $\nu_k^*$ gives an isomorphism between $A$ and $A_k$ for $k$ sufficiently large.

Recall the conductor ideal $\cond(B/A)\subset A$ is the annihilator of the $A$-module $B/A$. Similarly there are conductors $\cond(B/A_k)$ and $\cond(A_k/A)$. The conductor:
\[
\cond(B/A)\subset A\subset B
\]
is both an ideal for $A$ and $B$. The following hold:
\begin{itemize}
\item $\cond(B/A)\subset \cond(B/A_k)$,
\item any $B$-ideal $I\subset \cond(B/A)\subset A$ is also an $A$-ideal,
\item the cosupport of the ideal $\cond(B/A)$ in $B$ is set-theoretically the preimage of the non-normal locus of $X$, and 
\item for $k>0$ sufficiently large $\Oc(-kD) \subset \cond(B/A)$.
\end{itemize}

\noindent Assume that $k$ is large enough so there is containment. Then we have
\[
\Oc(-kD)=I_{kF_1\sqcup \cdots \sqcup kF_\ell}=I_{kF}\subset A\subset A_k\subset B.
\]
This gives rise to two commutative diagrams of $A$-modules.

\hspace{-.3in}\begin{tabular}{cc}
\begin{tikzcd}
&&0\arrow[d]&0\arrow[d]&\\
0\arrow[r]&A\arrow[r]\arrow[d,equal]&A_k\arrow[d]\arrow[r]&\Qc_k\arrow[d]\arrow[r]&0\\
0\arrow[r]&A\arrow[r]&B\arrow[r]\arrow[d]&\Qc\arrow[r]\arrow[d]&0\\
&&\Qc'\arrow[d]\arrow[r,equal]&\Qc'\arrow[d]&\\
&&0&0&
\end{tikzcd}&\begin{tikzcd}
&&0\arrow[d]&0\arrow[d]&\\
0\arrow[r]&A/I_{kF}\arrow[r]\arrow[d,equal]&S_k\arrow[d]\arrow[r]&\Qc_k\arrow[d]\arrow[r]&0\\
0\arrow[r]&A/I_{kF}\arrow[r]&R_k\arrow[r]\arrow[d]&\Qc\arrow[r]\arrow[d]&0\\
&&\Qc'\arrow[d]\arrow[r,equal]&\Qc'\arrow[d]&\\
&&0&0&
\end{tikzcd}\\
Diagram A. & Diagram B.
\end{tabular}

Diagram B is obtained by tensoring Diagram A with $A/I_{kF}$. Exactness in Diagram B relies on the containment $I_{kF}\subset \cond(B/A)\subset \cond(B/A_k)$. Therefore, there is $k>0$ such that
\[
\Supp(A_k/A) = \Supp(S_k/(A/I_{kF})).
\]
Finally,
\[
\Spec(A/I_{kF}) = kF\text{ and }\Spec(S_k) = kF_1\sqcup \cdots \sqcup kF_\ell
\]
are isomorphic away from the intersections of the $F_i$, which have codimension $\ge 2$ in $X$. It follows that
\[
\nu_k\cl X_k \ra X
\]
is a finite, proper, birational map that is isomorphism away from codimension 2 so we are done by Lemma \ref{lem:isoincodim2}.
\end{proof}

\begin{proof}[Proof of Corollary \ref{cor:pointaboveintpoint}]
Assume that the divisors $D_1$ and $D_2$ do not intersect over $x$. Then as the pushouts are also pushouts at the level of sets of $\CC$-points, the fiber of $\nu_k$ over $x$ has at least 2 points for all $k>0$ a contradiction.
\end{proof}

\begin{proof}[Proof of Corollary \ref{cor:branchesmustmeetinnorm}]
There is an \'etale neighborhood $\psi\cl U\ra X$ of $x\in X$ and a point $y\in U$ such that the analytic branches of $F$ correspond to irreducible components of $\psi^{-1}F$ going through $y$. Then apply Corollary \ref{cor:pointaboveintpoint}.
\end{proof}

\section{Intersection graphs}\label{sec:dualgraphs}

The purpose of this section is to define for any 1-dimensional scheme $C$ a bipartite graph $\Gamma(C)$. This is a generalization of the dual graph of a semistable curve that behaves reasonably well for S2 families of curves.

\begin{definition}
Let $C$ be a purely one-dimensional scheme. The \textit{intersection graph}, $\Gamma=\Gamma(C)$ of $C$ is a bipartite graph (with green and yellow vertices):
\begin{itemize}
\item the set of green vertices $G=G(C)$ correspond to one dimensional components of $C$,
\item the set of yellow vertices $Y=Y(C)$ correspond to points on $C$ where $C$ has multiple analytic branches, and
\item the set of edges $E=E(C)$ connecting a yellow vertex $y$ to a green vertex $g$ are in correspondence with the analytic branches of $g$ that go through $y$.
\end{itemize}
\end{definition}

\begin{remark}
\begin{enumerate}
\item The graph $\Gamma$ is bipartite as edges necessarily connect yellow vertices to green vertices.
\item If $C$ is a curve with only nodal singularities, then $\Gamma(C)$ is the barycentric subdivision of the dual graph of $C$.
\item Every yellow vertex has at least 2 adjacent edges.
\end{enumerate}
\end{remark}

\begin{center}

\begin{tabular}{cc}
\scalebox{1.2}{\begin{tabular}{cc}
\begin{tabular}{l}
\scalebox{2.5}{\begin{tikzpicture}[gren0/.style = {draw, circle,fill=greener!80,scale=.7},gren/.style ={draw, circle, fill=greener!80,scale=.4},blk/.style ={draw, circle, fill=black!,scale=.03}] 
\node[blk] at (-.1,.35) (1){};
\node[blk] at (.1,-.3) (2){};
\node[blk] at (-.4,-.15) (3){};
\node[blk] at (.2,-.15) (4){};
\node[blk] at (-.4,.2) (5){};
\node[blk] at (-.15,-.3) (6){};
\node[blk] at (-.25,.1) (7){};
\node[blk] at (-.3,-.3) (8){};
\node[scale=.35] at (-.5,0) (0){$C_1$};
\draw [-] (1) to [out=-100,in=100] (2);
\draw [-] (3) to [out=-15,in=190] (4);
\draw [-] (5) to [out=-80,in=130] (6);
\draw [-] (7) to [out=-80,in=55] (8);
\end{tikzpicture}}
\end{tabular}
&
\begin{tabular}{l}
\begin{tikzpicture}[gren0/.style = {draw, circle,fill=greener!80,scale=.7},gren/.style ={draw, circle, fill=greener!80,scale=.4}]
\node[gren0] at (0,0) (1){};
\node[gren0] at (.6,.6) (2){};
\node[scale=.5] at (.3,.3) (3){};
\node[scale=.5] at (-.3,.3) (4){};
\node[gren0] at (-.45,.7) (5){};
\node[gren0] at (-.75,.5) (6){};
\node[scale=.8] at (1.3,.2) (0){$\Gamma(C_1)$};
\draw (1)--(3);
\draw (2)--(3);
\draw (1)--(4);
\draw (4)--(5);
\draw (4)--(6);
\filldraw [orange!60,scale=.11] (2.7,2.7) circle () {};
\filldraw [orange!60,scale=.11] (-2.7,2.7) circle () {};
\end{tikzpicture}
\end{tabular}
\end{tabular}}\hspace{1cm}&\scalebox{1.2}{\begin{tabular}{ccc}
\begin{tabular}{l}
\scalebox{2.5}{\begin{tikzpicture}[gren0/.style = {draw, circle,fill=greener!80,scale=.7},gren/.style ={draw, circle, fill=greener!80,scale=.4},blk/.style ={draw, circle, fill=black!,scale=.03}]
\node[blk] at (0,0) (1){}; 
\node[blk] at (0,.2) (2){};
\node[blk] at (-.1,.4) (3){};
\node[blk] at (.1,-.3) (4){};
\node[blk] at (-.4,-.15) (5){};
\node[blk] at (.2,-.15) (6){};
\node[blk] at (-.25,.1) (7){};
\node[blk] at (-.25,-.3) (8){};
\node[scale=.35] at (-.5,0) (0){$C_2$};
\draw [-] (3) to [out=-100,in=135] (2);
\draw [-] (2) to [out=-135,in=135] (1);
\draw [-] (1) to [out=-135,in=100] (4);
\draw [-] (5) to [out=-15,in=190] (6);
\draw [-] (7) to [out=-70,in=70] (8);
\draw [-] (2) to [out=-45,in=45,looseness=50] (2);
\draw [-] (1) to [out=-45,in=45,looseness=50] (1);
\end{tikzpicture}} \end{tabular} &
\begin{tabular}{l}
\begin{tikzpicture}[gren0/.style = {draw, circle,fill=greener!80,scale=.7},gren/.style ={draw, circle, fill=greener!80,scale=.4}]
\node[gren0] at (0,0) (1){};
\node[gren0] at (.6,.6) (2){};
\node[scale=.5] at (.3,.3) (3){};
\node[scale=.5] at (-.3,.3) (4){};
\node[gren0] at (-.6,.6) (5){};
\node[scale=.5] at (1,.7) (6){};
\node[scale=.5] at (.3,.8) (7){};
\node[scale=.8] at (1.3,.2) (0){$\Gamma(C_2)$};
\draw (1)--(3);
\draw (2)--(3);
\draw (1)--(4);
\draw (4)--(5);
\draw [-] (2) to [out=-30,in=-120] (6);
\draw [-] (2) to [out=50,in=150] (6);
\draw [-] (2) to [out=180,in=-80] (7);
\draw [-] (2) to [out=110,in=10] (7);
\filldraw [orange!60,scale=.11] (9,6.3) circle () {};
\filldraw [orange!60,scale=.11] (2.7,7.5) circle () {};
\filldraw [orange!60,scale=.11] (9,6.3) circle () {};
\filldraw [orange!60,scale=.11] (2.7,2.7) circle () {};
\filldraw [orange!60,scale=.11] (-2.7,2.7) circle () {};
\end{tikzpicture}
\end{tabular}
\end{tabular}}
\end{tabular}
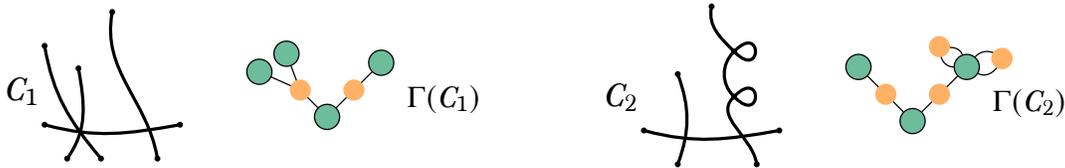
Figure. Examples of curves and their intersection graphs.
\end{center}

\begin{definition}
There are two natural maps $E\ra G$ and $E\ra Y$. For $g\in G$ (resp. $y\in Y$) we use $E_g$ (resp. $E_y\in Y$) to denote the \textit{edges adjacent to $g\in G$} (resp. the \textit{edges adjacent to $y\in Y$}): these are the fibers of these maps. Likewise, for any $g\in G$, we define the set of \textit{adjacent vertices} to be:
\[
Y_g=\{ y\in Y | g\text{ and }y\text{ are adjacent} \subset \Gamma\} \subset Y.
\]
Similarly for any $y\in Y$ we define $G_y\subset G$.
\end{definition}

\begin{definition}
Given a \textbf{finite} map of curves $f\cl C_1\ra C_2,$ the \textit{induced map of graphs\footnote{This is an abuse of terminology and is really a map of the underlying topological spaces.}}
\[
\Gamma_f\cl \Gamma(C_1)\ra \Gamma(C_2),
\]
is defined as follows:

\begin{enumerate}
\item For $g\in G(C_1)$, the corresponding component maps to a unique component in $G(C_2)$,
\item For $y\in Y(C_1)$ there are two cases:
\begin{enumerate}
\item If all of the branches of $C_1$ at $y$ map to the same branch in $C_2$ then that branch determines a unique irreducible component $g'\in G(C_2)$. Set $\Gamma_f(y) = g'$.
\item If the branches of $C_1$ at $y$ map to more than one branch of $C_2$, then $y'=f(y)\in C_2$ is a point with multiple analytic branches. Set $y'=\Gamma_f(y)$.
\end{enumerate}
\item For $e\in E(C_1)$ that connects the vertex $g\in G(C_1)$ to $y\in Y(C_1)$:
\begin{enumerate}
\item If $\Gamma_f(y) = \Gamma_f(g)=g'$ then $\Gamma_f(e) := g'$.
\item Otherwise $e$ determines a unique branch of $f(g)$ at $f(y) = y'$ and $\Gamma_f(e)$ is defined to be that branch.
\end{enumerate}
\end{enumerate}
\end{definition}

\begin{definition}
Let $f\cl C_1\ra C_2$ be a finite map of curves. 

\begin{enumerate}
\item For any yellow vertex  $y \in \Gamma(C_1)$ such that $\Gamma_f(y)$ is yellow, we define the \textit{branch injectivity failure} of $\Gamma_f$ at $y$ be
\[
\zeta_f(y) := \#(E_y)-\#(\Gamma_f(E_y)).
\]
(If $y$ maps to a green vertex, set $\zeta_f(y)=0$.)
\item For $g\in G(C_g)$ with associated component $C'\subset C_1$ we define the \textit{component injectivity failure} of $\Gamma_f$ at $g$ to be the number
\[
\zeta_f(g) := \deg\left(f|_{(C_g)^\mathrm{red}}\right) - 1.
\]
Thus $\zeta_f(g) = 0 \iff f$ maps $(C_g)^\mathrm{red}$ birationally onto its image.
\item We say \textit{$\Gamma_f$ is locally injective at $y$} if $\zeta_f(y) = 0$ and $\zeta_f(g) = 0$ for all $g\in E_y$.
\item For any curve $C$, we define the \textit{multiplicity} $m(g)$ of $g\in G(C)$ to be the length of the local ring of $C$ at the generic point of the associated curve $C_g$.
\end{enumerate}
\end{definition}

\begin{lemma}\label{deltaif}
Let $C_1$ be a reduced curve and let $f\cl C_1\ra C_2$ be a finite proper map of curves. If there is a yellow vertex $y\in Y(C_1)$ such that $\Gamma_f$ is locally injective at $y$ then there are analytic neighborhoods $y\in \Delta_1\subset C_1$ and $f(y) \in \Delta_2\subset f(C_1)^\mathrm{red}$ such that the map
\[
\Delta_1\ra \Delta_2
\]
is a partial normalization. As a consequence we have an inequality of $\delta$-invariants:
\[
\delta_y(C_1) \le \delta_{f(y)}(\Delta_2).
\]
\end{lemma}

\begin{proof}
Any proper birational map is a partial normalization. As the branches at $y$ map injectively to the branches at $f(y)$ and each curve adjacent to $y$ maps birationally onto its image we can take a small enough neighborhood where the map is birational and proper.
\end{proof}

\begin{definition}
Let $\Gamma(C)$ be the intersection graph of a curve. Let $y\in \Gamma(C)$ be a yellow vertex. We define the \textit{local Euler characteristic} at $y$ to be:
\[
\chi_\loc(y) = 1-\#(E_y).
\]
\end{definition}

Note that, for any yellow vertex $y\in \Gamma(C)$ we have $\chi_\loc(y)\le -1$. The following is a straightforward application of bipartiteness:
\begin{equation}\label{eulercharequality}
\chi(\Gamma(C)) = \#(G(C))+\sum_{y\in \Gamma(C)}\chi_\loc(y)
\end{equation}

Now we consider the case of interest to us. Let
\[
X\ra T\ni 0
\]
be a flat map from an irreducible S2 surface $X$ to a smooth pointed curve $0\in T$ such that the non-normal locus of $X$ is contained in the central fiber $X_0$. Let
\[
\nu\cl X^\nu \ra X
\]
denote the normalization map.

\begin{theorem}\label{normalizationdualgraphproperties}
Let $X_T$ be an S2 surface as above. Assume that the curve $X^\nu_0$ is reduced.
\begin{enumerate}
\item The induced map $\Gamma_\nu\cl \Gamma(X^\nu_0)\ra \Gamma(X_0)$ is surjective.
\item There is an inequality:
\begin{align*}
\chi(\Gamma(X^\nu_0))\le \chi(\Gamma(X_0))&+\sum_{g\in G(X_0)} \left(m(g)-1\right) -\sum_{g\in G(X_0^\nu)} \zeta_\nu(g) -\sum_{y\in Y(X_0^\nu)} \zeta_\nu(y).
\end{align*}
\item Defining
\[
M:=\sum_{g\in G(X_0)} (m(g)-1).
\]
If $\chi(\Gamma^\nu_0) = \chi(\Gamma(X_0))+M$ then $\zeta_\nu(y) =\zeta_\nu(g)= 0$ for all $y\in Y(X_0^\nu)$ and $g\in G(X_0^\nu)$.
\end{enumerate}
\end{theorem}

Roughly speaking this says that the difference in the Euler characteristics of the intersection graphs of $X_0^\nu$ and $X_0$ are controlled by the multiplicities of $X_0$.

\begin{proof}
To prove part (1), note that it is clear for green vertices in $\Gamma(X_0)$. For any yellow vertex $y\in \Gamma(X_0)$ and any edge $e\in E_y$, partition the branches at $y$ as follows:
\[
E_y = \{e\} \sqcup (E_y\setminus \{e\}).
\]
Then the one dimensional pre-images of the branches $\{e\}$ and $E_y\setminus \{e\}$ must intersect by Corollary \ref{cor:branchesmustmeetinnorm}. This shows $e\in \Gamma_f(E(X^\nu_0))$ and it follows that $y\in \Gamma_f(Y(X^\nu_0)).$

For (2), we would like to use Equation \ref{eulercharequality}. As the pushforward of the cycle $[X^\nu_0]$ is $[X_0]$ we have:
\[
\sum_{g\in G(X_0)} m(g) = \sum_{g\in G(X_0^\nu)} (\zeta_\nu(g)+1).
\]
Which gives
\[
\#G(X_0)+\sum_{g\in G(X_0)} (m(g)-1) = \left(\sum_{g\in G(X_0^\nu)} \zeta_\nu(g)\right)+\#G(X_0^\nu).
\]
This shows:
\[
\#G(X_0^\nu) = \#G(X_0) + \sum_{g\in G(X_0)}(m(g)-1) - \sum_{g\in G(X_0^\nu)} \zeta_\nu(g). 
\]

To complete the proof of (2) we need to prove the inequality:
\[
\sum_{y\in Y(X_0^\nu)} \left(\chi_\loc(y) + \zeta_\nu(y)\right)\le \sum_{y\in Y(X_0)} \chi_\loc(y).
\]
For any yellow vertex $y\in \Gamma(X_0^\nu)$, we have $\chi_\loc(y)\le -1.$ Thus if $\Yt(X_0^\nu)\subset Y(X_0^\nu)$ represents the yellow vertices that map to yellow vertices:
\[
\sum_{y\in Y(X_0^\nu)} \left(\chi_\loc(y) + \zeta_\nu(y)\right)\le\sum_{y\in \Yt(X_0^\nu)} \left(\chi_\loc(y) + \zeta_\nu(y)\right).
\]
So for each $y'\in Y(X_0)$ it suffices to check that the inequality
\[
\sum_{y\in \Gamma_\nu^{-1}(y)} \left(\chi_\loc(y) +\zeta_\nu(y)\right) \le \chi_\loc(y').
\]

First we show that we can order the vertices
\[
\Gamma_\nu^{-1}(y')=\{y_1<y_2<\cdots<y_\ell\}
\]
such that for each $y_i$ (with $i>1$) the intersection
\[
\emptyset \ne \Gamma_\nu(E_{y_i})\bigcap \left(\bigcup\limits_{y_j<y_i} \Gamma_\nu(E_{y_j})\right)\subset E_{y'},
\]
i.e. there is an overlap in the images of the edges adjacent to $y_i$ and the previous $y_j$s.

We proceed by induction. Suppose, we know it up to step $(i-1)$. Define:
\[
Y_{i-1}:=\{y_1,\dots,y_{i-1}\}\text{ and }E_{i-1} := \bigcup_{y\in Y_{i-1}} E_{y_i}.
\]
Suppose for contradiction that for all $y_j\in \Gamma_\nu^{-1}(y')\setminus Y_{i-1}$ the intersection:
\[
\Gamma_{\nu}(E_{i-1})\cap \Gamma_{\nu}(E_{y_j}) = \emptyset.
\]
Then we can partition the branches $g\in E_{y'}$ into two sets:
\[
E_{y'} = \Gamma_\nu(E_i) \sqcup \left(E_{y'}\setminus \Gamma_\nu(E_i)\right)
\]
Therefore the one-dimensional preimages of the branches in $\Gamma_\nu(E_i)$ and $E_{y'}\setminus \Gamma_\nu(E_i)$ do not intersect, which contradicts Corollary \ref{cor:pointaboveintpoint}.

Thus we can order the yellow vertices as desired. Now we can count:
\begin{align*}
\chi_\loc(y_1)&\le 1-\#(\Gamma_\nu(E_1))-\zeta(y_1),\\
\chi_\loc(y_1)+\chi_\loc(y_2)&\le 1-\#(\Gamma_\nu(E_2))-\zeta(y_1)-\zeta(y_2),\\
&\dots\\
\chi_\loc(y_1)+\cdots+\chi_\loc(y_\ell)&\le \chi_\loc(y')-\zeta(y_1)-\cdots-\zeta(y_\ell).
\end{align*}
\noindent These inequalities can be proved in order, using the property of the ordering. This completes the proof of (2). And the proof of (3) follows easily.
\end{proof}

Now we would like to classify the possible intersection graphs that appear in Hacking's Calabi-Yau limits of degree 7 curves and degree $5$ curves. Let $\Dcy$ be an S2 surface with a projective map to a curve $T$ such that the nonnormal locus is contained in $\Dcyo$. Let $\Dnorm$ be the normalization of $\Dcy$ and assume that $\Dnormo$ is reduced.

If it $D$ was a limit of degree 7 curves then the hypotheses (H1)-(H6) below are satisfied, and if $D$ was a limit of degree 5 curves then the hypotheses (H1$\star$)-(H6$\star$) below are satisfied (see Lemma~\ref{lem:Dnormsatisfies(Hi)}).
\begin{enumerate}[leftmargin=6\parindent]
\item[(H1)=(H1$\star$)] The intersection graph, $\Gamma(\Dnormo)$ is a tree.
\item[(H2)=(H2$\star$)] If $\Dcyo$ is nonreduced then there is some vertex $v\in G(\Dnormo)$ or $v\in Y(\Dnormo)$ such that $\zeta_\nu(v)>0$.
\item[(H3)] The multiplicity of any curve in $\Dcyo$ is at most 2.
\item[(H3$\star$)] The curve $\Dcyo$ is reduced.
\item[(H4)] At any point in $\Dcyo$, the multiplicity is at most 4.
\item[(H4$\star$)] At any point in $\Dcyo$, the multiplicity is at most 3.
\item[(H5)=(H5$\star$)] Any two green vertices $g_1,g_2\in\Gamma(\Dcyo)$ have graph distance 2.
\item[(H6)=(H6$\star$)] At least one component of $\Dcyo$ is reduced.
\end{enumerate}

\begin{theorem}\label{thm:possibledualgraphs}
With the hypotheses (H1)-(H6) above, $\Dcyo$ has between 1 and 4 components and has one of the following intersection graphs $\Gamma$ (the green vertices marked with 2s are doubled curves):
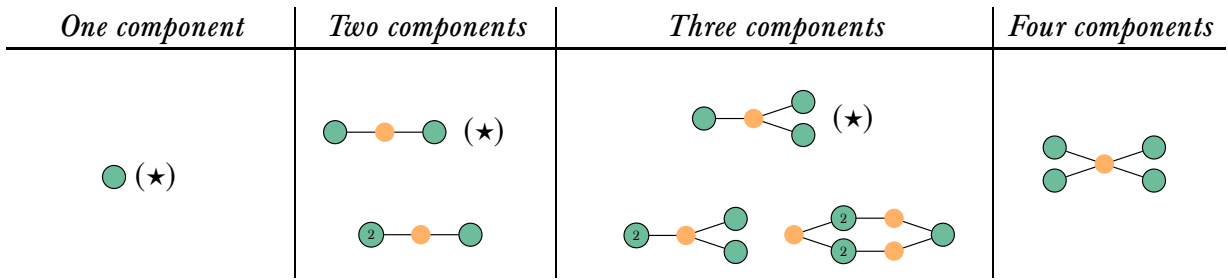
\begin{center}
\begin{tabular}{c|c|c|c}
\hspace{.5cm}One component\hspace{.5cm} & Two components & Three components & Four components\\
\hline
&&\\
\scalebox{1.1}{
\begin{tabular}{c}
\begin{tikzpicture}[gren0/.style = {draw, circle,fill=greener!80,scale=.7},gren/.style ={draw, circle, fill=greener!80,scale=.4}]
\node[gren0] at (0,0) (1){};
\node at (.5,0) {\small$(\star)$};
\end{tikzpicture}

\end{tabular}} & \scalebox{1.1}{\begin{tabular}{c}
\begin{tikzpicture}[gren0/.style = {draw, circle,fill=greener!80,scale=.7},gren/.style ={draw, circle, fill=greener!80,scale=.4}]
\node[gren0] at (-.6,0) (1){};
\node[scale=.5] (1) at (0,0) (2){};
\node[gren0] at (.6,0) (3){};
\node at (1.2,0){\small$(\star)$};
\draw (1)--(2);
\draw (2)--(3);
\filldraw [orange!60,scale=.11] (0,0) circle () {};
\end{tikzpicture}
\\
\\
\begin{tikzpicture}[gren0/.style = {draw, circle,fill=greener!80,scale=.7},gren/.style ={draw, circle, fill=greener!80,scale=.4}]
\node[gren] at (-.6,0) (1){$2$};
\node[scale=.5] (1) at (0,0) (2){};
\node[gren0] at (.6,0) (3){};
\draw (1)--(2);
\draw (2)--(3);
\filldraw [orange!60,scale=.11] (0,0) circle () {};
\end{tikzpicture}

\end{tabular}} & \hspace{.4cm}
\scalebox{1.1}{\begin{tabular}{c}
\begin{tikzpicture}[gren0/.style = {draw, circle,fill=greener!80,scale=.7},gren/.style ={draw, circle, fill=greener!80,scale=.4}] 
\node[scale=.5] (1) at (0,0) (1){};
\node[gren0] at (-.6,0) (2){};
\node[gren0] at (.6,-.2) (3){};
\node[gren0] at (.6,.2) (4){};
\node at (1.2,0) {\small$(\star)$};
\draw (1)--(3);
\draw (1)--(2);
\draw (1)--(4);
\filldraw [orange!60,scale=.11] (0,0) circle () {};
\end{tikzpicture} 
\\
\\
\begin{tikzpicture}[gren0/.style = {draw, circle,fill=greener!80,scale=.7},gren/.style ={draw, circle, fill=greener!80,scale=.4}] 
\node[scale=.5] (1) at (0,0) (1){};
\node[gren] at (-.6,0) (2){$2$};
\node[gren0] at (.6,-.2) (3){};
\node[gren0] at (.6,.2) (4){};
\draw (1)--(3);
\draw (1)--(2);
\draw (1)--(4);
\filldraw [orange!60,scale=.11] (0,0) circle () {};
\end{tikzpicture} 
\hspace{.2cm}

\begin{tikzpicture}[gren0/.style = {draw, circle,fill=greener!80,scale=.7},gren/.style ={draw, circle, fill=greener!80,scale=.4}] 
\node[scale=.5] at (0,0) (1){};
\node[gren] at (.6,-.2) (2){$2$};
\node[gren] at (.6,.2) (3){$2$};
\node[scale=.5] at (1.2,-.2) (4){};
\node[scale=.5] at (1.2,.2) (5){};
\node[gren0] at (1.8,0) (6){};
\draw (1)--(2);
\draw (1)--(3);
\draw (2)--(4);
\draw (3)--(5);
\draw (4)--(6);
\draw (5)--(6);
\filldraw [orange!60,scale=.11] (11,1.8) circle () {};
\filldraw [orange!60,scale=.11] (11,-1.8) circle () {};
\filldraw [orange!60,scale=.11] (0,0) circle () {};
\end{tikzpicture}
\end{tabular}} & \scalebox{1.1}{\begin{tikzpicture}[gren0/.style = {draw, circle,fill=greener!80,scale=.7},gren/.style ={draw, circle, fill=greener!80,scale=.4}] 
\node[scale=.5] (1) at (0,0) (1){};
\node[gren0] at (-.6,-.2) (2){};
\node[gren0] at (-.6,.2) (3){};
\node[gren0] at (.6,-.2) (4){};
\node[gren0] at (.6,.2) (5){};
\draw (1)--(3);
\draw (1)--(2);
\draw (1)--(4);
\draw (1)--(5);
\filldraw [orange!60,scale=.11] (0,0) circle () {};
\end{tikzpicture}} 
\end{tabular}
\end{center}
With the hypotheses (H1$\star$)-(H6$\star$), $\Dcyo$ has between 1 and 3 components and its intersection graph is one of the graphs marked with a star.
\end{theorem}

\begin{proof}
Throughout let $|G|$ denote the number of green vertices. We will use the following fact:

\begin{addmargin}[1em]{2em}
\textbf{Fact}: If $\Gamma$ is a connected graph and $\Gamma'\subset \Gamma$ is any subgraph, then $\chi(\Gamma)\le \chi(\Gamma').$
\end{addmargin}

To start assume $\Dcyo$ is reduced; so $\Dcy = \Dnorm$. By (H1), $\Gamma$ is a tree. Assume that there are at least 2 green vertices. Consider the graph $\Gamma'$ obtained by deleting one of the green vertices $g\in \Gamma$ (and all the adjacent edges $E_g$). By (H5) this graph is still connected, so has $\chi(\Gamma') = 1$. Thus $1=\chi(\Gamma) = \chi(\Gamma')+1-(\#(E_g))$. Thus there must be one edge adjacent to any $g \in \Gamma$, and therefore $\Gamma$ is a star graph with a single yellow vertex (these are the top graphs in the table). Notice, that by (H4) there can be at most 4 green vertices in $\Gamma$.

Now assume that $\Dcyo$ is nonreduced. We organize our proof by the number of doubled curves $M\ge 1$ in $\Dcyo$. By Theorem~\ref{normalizationdualgraphproperties} and assumptions (H1) and (H2), we have:
\begin{equation}
\chi(\Gamma)+M\ge 2.
\end{equation}
To start, consider the case $M=1$. Thus:
\[
\chi(\Gamma)\ge 1
\]
so $\Gamma(\Dcyo)$ is a tree. The argument from the reduced case implies it is a star with center a yellow vertex. By (H4) there are at most three components (using that now there is a doubled curve). This gives two possible graphs in the table.

Now assume $M=2$. By (H6) we have $|G| \ge 3$. If $|G|=3$ then by (H4) and (H5), $\Gamma$ contains the following subgraph.

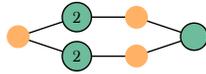
\begin{figure}[h]
\centering
\scalebox{1.3}{\begin{tikzpicture}[gren0/.style = {draw, circle,fill=greener!80,scale=.7},gren/.style ={draw, circle, fill=greener!80,scale=.4}] 
\node[scale=.5] at (0,0) (1){};
\node[gren] at (.6,-.2) (2){$2$};
\node[gren] at (.6,.2) (3){$2$};
\node[scale=.5] at (1.2,-.2) (4){};
\node[scale=.5] at (1.2,.2) (5){};
\node[gren0] at (1.8,0) (6){};
\draw (1)--(2);
\draw (1)--(3);
\draw (2)--(4);
\draw (3)--(5);
\draw (4)--(6);
\draw (5)--(6);
\filldraw [orange!60,scale=.11] (11,1.8) circle () {};
\filldraw [orange!60,scale=.11] (11,-1.8) circle () {};
\filldraw [orange!60,scale=.11] (0,0) circle () {};
\end{tikzpicture}}
\end{figure}

\noindent This has Euler characteristic 0, so there cannot be any additional yellow vertices or edges. If $|G|\ge 4$ then $\Gamma$ contains one of the following subgraphs:

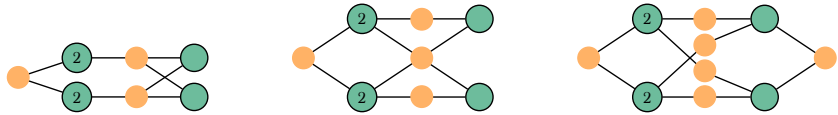
\begin{figure}[h]
\centering
\scalebox{1.3}{\begin{tikzpicture}[gren0/.style = {draw, circle,fill=greener!80,scale=.7},gren/.style ={draw, circle, fill=greener!80,scale=.4}] 
\node[scale=.5] at (0,0) (1){};
\node[gren] at (.6,-.2) (2){$2$};
\node[gren] at (.6,.2) (3){$2$};
\node[scale=.5] at (1.2,-.2) (4){};
\node[scale=.5] at (1.2,.2) (5){};
\node[gren0] at (1.8,.2) (6){};
\node[gren0] at (1.8,-.2) (7){};
\draw (1)--(2);
\draw (1)--(3);
\draw (2)--(4);
\draw (3)--(5);
\draw (4)--(6);
\draw (5)--(6);
\draw (4)--(7);
\draw (5)--(7);
\filldraw [orange!60,scale=.11] (11,1.8) circle () {};
\filldraw [orange!60,scale=.11] (11,-1.8) circle () {};
\filldraw [orange!60,scale=.11] (0,0) circle () {};
\end{tikzpicture}

\hspace{.6cm}

\begin{tikzpicture}[gren0/.style = {draw, circle,fill=greener!80,scale=.7},gren/.style ={draw, circle, fill=greener!80,scale=.4}] 
\node[scale=.5] at (0,0) (1){};
\node[gren] at (.6,-.4) (2){$2$};
\node[gren] at (.6,.4) (3){$2$};
\node[scale=.5] at (1.2,-.4) (4){};
\node[scale=.5] at (1.2,.4) (5){};
\node[scale=.5] at (1.2,0) (6){};
\node[gren0] at (1.8,-.4) (7){};
\node[gren0] at (1.8,.4) (8){};
\draw (1)--(2);
\draw (1)--(3);
\draw (2)--(4);
\draw (3)--(5);
\draw (2)--(6);
\draw (3)--(6);
\draw (4)--(7);
\draw (6)--(7);
\draw (6)--(8);
\draw (5)--(8);
\filldraw [orange!60,scale=.11] (11,3.6) circle () {};
\filldraw [orange!60,scale=.11] (11,-3.6) circle () {};
\filldraw [orange!60,scale=.11] (11,0) circle () {};
\filldraw [orange!60,scale=.11] (0,0) circle () {};
\end{tikzpicture}

\hspace{.6cm}

\begin{tikzpicture}[gren0/.style = {draw, circle,fill=greener!80,scale=.7},gren/.style ={draw, circle, fill=greener!80,scale=.4}] 
\node[scale=.5] at (0,0) (1){};
\node[gren] at (.6,-.4) (2){$2$};
\node[gren] at (.6,.4) (3){$2$};
\node[scale=.5] at (1.2,-.4) (4){};
\node[scale=.5] at (1.2,.4) (5){};
\node[scale=.5] at (1.2,-.15) (6){};
\node[scale=.5] at (1.2,.15) (7){};
\node[gren0] at (1.8,-.4) (8){};
\node[gren0] at (1.8,.4) (9){};
\node[scale=.5] at (2.4,0) (10){};
\draw (1)--(2);
\draw (1)--(3);
\draw (2)--(4);
\draw (3)--(5);
\draw (2)--(7);
\draw (3)--(6);
\draw (6)--(8);
\draw (4)--(8);
\draw (7)--(9);
\draw (5)--(9);
\draw (8)--(10);
\draw (9)--(10);
\filldraw [orange!60,scale=.11] (10.8,3.6) circle () {};
\filldraw [orange!60,scale=.11] (10.8,1.2) circle () {};
\filldraw [orange!60,scale=.11] (10.8,-1.2) circle () {};
\filldraw [orange!60,scale=.11] (10.8 ,-3.6) circle () {};
\filldraw [orange!60,scale=.11] (0,0) circle () {};
\filldraw [orange!60,scale=.11] (22,0) circle () {};
\end{tikzpicture}}
\end{figure}

\noindent These all have Euler characteristic at most $-1$ which is a contradiction.

If $M\ge 3$ then by (H6) $|G|\ge M+1$. By (H4) any yellow vertex that meets two doubled vertices cannot meet any other vertex. By (H5) there are at least $\binom{M}{2}$ yellow vertices connecting pairs of doubled components and at least $M$ yellow vertices connecting the doubled components to some reduced component. For example when $M=3$, $\Gamma$ contains the subgraph:
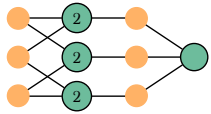
\begin{figure}[H]
\centering
\scalebox{1.3}{\begin{tikzpicture}[gren0/.style = {draw, circle,fill=greener!80,scale=.7},gren/.style ={draw, circle, fill=greener!80,scale=.4}] 
\node[scale=.5] (1) at (0,-.4) (1){}; 
\node[scale=.5] (1) at (0,.4) (2){};
\node[scale=.5] (1) at (0,0) (3){};
\node[gren] at (.6,-.4) (4){$2$};
\node[gren] at (.6,.4) (5){$2$};
\node[gren] at (.6,0) (6){$2$};
\node[scale=.5] at (1.2,-.4) (7){};
\node[scale=.5] at (1.2,.4) (8){};
\node[scale=.5] at (1.2,0) (9){};
\node[gren0] at (1.8,0) (10){};
\draw (1)--(6);
\draw (2)--(5);
\draw (1)--(4);
\draw (3)--(5);
\draw (2)--(6);
\draw (3)--(4);
\draw (4)--(7);
\draw (5)--(8);
\draw (6)--(9);
\draw (7)--(10);
\draw (8)--(10);
\draw (9)--(10);
\filldraw [orange!60,scale=.11] (0,3.6) circle () {};
\filldraw [orange!60,scale=.11] (0,-3.6) circle () {};
\filldraw [orange!60,scale=.11] (0,0) circle () {};
\filldraw [orange!60,scale=.11] (11,3.6) circle () {};
\filldraw [orange!60,scale=.11] (11,-3.6) circle () {};
\filldraw [orange!60,scale=.11] (11,0) circle () {};
\end{tikzpicture}}
\end{figure}
\noindent Such a graph (assuming $M\ge 3$) has
\[
\chi(\Gamma) \le M+1-\binom{M}{2}-M = 1-\binom{M}{2}\le 1-M.
\]
This is too negative.

Replacing (H3) by (H3$\star$) and (H4) by (H4$\star$) selects the starred graphs in the table.
\end{proof}

\begin{lemma}\label{lem:Dnormsatisfies(Hi)}
Let $D$ is a family of smooth projective curves over a smooth curve $T$. If the general fiber is a degree 7 plane curve then after possibly making a change of base $\Dnormo$ is reduced and hypotheses (H1)-(H6) are satisfied by $\Dcy$ and its normalization $\Dnorm$. Similarly, if $D$ is a family of degree 5 curves then after possibly making a change of base, $\Dnormo$ is reduced and hypotheses (H1$\star$)-(H6$\star$) are satisfied.
\end{lemma}

\begin{proof}
After possible making a change of base, we can assume that there is a resolution of singularities $\Dss$ of $\Dcy$ with reduced, nodal central fiber. The family $D$ is the relative minimal model of the family, so there is a map $\Dss\ra D$ that contracts trees of rational curves. It follows that the intersection graph of $\Dsso$ is a tree and the same is true for $\Dnormo$ as $\Dnorm$ is a contraction of $\Dss$. Thus (H1) is satisfied.

Now suppose $\Dcyo$ is nonreduced. If (H2) does not hold, then either $D_0$ is birational to a component of $\Dcyo$ or $D_0$ is contracted in $\Dnorm$ (in which case, by Lemma~\ref{deltabranches} and Lemma~\ref{deltaif} there is a point $p\in(\Dcyo)^\red$ such that $\delta_p((\Dcyo)^\red)\ge (\#\text{ of branches})+g-1$). Either of these cannot happen as the arithmetic genus of $(\Dcyo)^\red$ is strictly smaller than the arithmetic genus of $\Dcyo$.

(H3) and (H4) are both consequences the log canonical threshhold bound:
\[
\lct(X_0,\Dcyo)\ge 3/7.
\]
(The log canonical threshhold of a multiplicity $a$ point is bounded from above by $2/a$. The log canonical threshhold of a multiplicity $b$ curve is bounded above by $1/b$.) (H5) uses that $\PP^2$ and $\PP(1,1,4)$ both have Picard rank 1.

For (H6), it is clear in the case $\Dcyo\subset \PP^2$ (as 7 is not a multiple of 2). In the case $\Dcyo\subset \PP(1,1,4)$, we apply Lemma~\ref{remainders}.

For degree 5 curves, the hypotheses can be checked similarly.
\end{proof}

\begin{lemma}\label{remainders}
Let $C=\sum m_i C_i\subset \PP(1,1,4)$ be a divisor where $C_i$ is a curve of degree $d_i$. If $r_i$ is the remainder of $d_i/4$ then the multiplicity of $C$ at the vertex in $\PP(1,1,4)$ is at least $\sum m_ir_i$. In particular, if $C\in |\Oc(14)|$ and $\lct(\PP(1,1,4),C)\ge 3/7$ then at least one component of $C$ is reduced.
\end{lemma}

\begin{proof}
Denote the coordinates on $\PP(1,1,4)$ by $x,y,z$.  For each $C_i$, we can write the equation of $C_i$ as a polynomial of the form $\sum_{j = 0}^{\lfloor d_i/4 \rfloor} z^j f_{d_i - 4j}(x,y)$.  In the local chart where $z \ne 0$, the curve is defined by a polynomial in $x,y$ whose minimal degree is $r_i$, so must pass through the vertex with multiplicity as least $r_i$.  Therefore, the curve $C$ must pass through the vertex with multiplicity $\sum m_i r_i$.  

If $C \in |\Oc(14)|$ and every component of $C$ is non-reduced, then $C = 2C'$, where $C'$ is a degree 7 curve that is possibly reducible.  By the previous paragraph, such a doubled curve has multiplicity at least $6$ at the vertex, which implies that the log canonical threshold is less that $\frac{3}{7}$.  
\end{proof}

Finally we give a quick corollary of Theorem~\ref{normalizationdualgraphproperties} regarding the possible dual graphs of stable replacements of curves in a linear series on a surface. Let $S$ be a smooth projective complex surface. Let $\Lambda$ be a linear series on $S$ such that the general member of $\Lambda$ is smooth with genus $g\ge 2$.

\begin{definition}
As in the previous paragraph, let $C=m_1C_1+\cdots+m_\ell C_\ell$ be a curve with $[C]\in \Lambda$. We say that a stable curve $D'$ is a \textit{stable replacement of $C$} if there is a smooth one dimensional pointed curve $0\in T$, a map $f\cl T\ra \Lambda$, and a family of curves $D$ over $T$ such that
\begin{enumerate}
\item $D_0 = D'$,
\item if $t\ne 0$ the curve $D_t$ is smooth and isomorphic to the curve defined by $f(t)\in \Lambda$, and
\item $f(0) = [C]\in \Lambda$.
\end{enumerate}
\end{definition}

Theorem~\ref{thmF:intgraphs} is immediate from the following corollary.

\begin{corollary}
With $C$ such that $[C]\in \Lambda$ as above, if $D'$ is a stable replacement of $C$, then
\[
\chi(\Gamma(D')) \le \chi(\Gamma(C))+\sum (m_i-1).
\]
\end{corollary}

\begin{remark}
In particular, this shows that the number of loops in the dual graph of $D$ (which is homotopic to $\Gamma(D)$) can be bounded from below by the number of loops in the dual graph of $C$ and the multiplicities of $C$. This is most interesting when $C$ is nonreduced.
\end{remark}

\begin{proof}
Let $T\ra \Lambda$ as above. Let $\Cc_T$ denote the pullback of the universal curve over $\Lambda$ to $T$. Note, that making a base change of $T$ does not change the stable limit at $0$, so after a base change we may assume that there is a resolution of singularities $\Cc'_T\ra\Cc_T$ such that the central fiber $\Cc'_0$ is reduced. Then $\Cc'_0$ is a semistable model of $D$, so they have homotopic dual graphs. The surface $\Cc_T\subset T\times S$ is S2 as it is a Cartier divisor in a smooth variety. The result then follows from Theorem~~\ref{normalizationdualgraphproperties}(3).
\end{proof}

\section{Background on curve singularities and Hacking's Calabi-Yau limits}\label{sec:curvebackground}

The goal of this section is to recall the background on singularities of curves in surfaces that is needed to prove Theorems \ref{thmD:deg5} and \ref{thmE:deg7} (i.e. to show that every smooth projective limit of a degree 5 curve is planar or hyperelliptic in $M(5)$, and show that every smooth projective limit of a degree 7 curve is planar).

Let $D$ be a smooth projective family of curves such that the general fiber is a plane curve and consider the associated family of pairs $(X_T,D)$ as in \S\ref{sec:CYlimits}. Let $(\Xcyo,\Dcyo)$ be Hacking's Calabi-Yau limit.  By Remark~\ref{DcyisS2}, we know that the family $\Dcy$ is S2. If the general fiber is a quintic curve, by Proposition \ref{degree5surfaces}, $\Xcyo = \PP^2$, $\PP(1,1,4)$, $M(5)$, or $\PP(1,4,25)$, and by Theorem~\ref{thm:possibledualgraphs} and Lemma~\ref{degree7noindex5surface} there are only three possibilities for the intersection graph of $\Dcyo$. Likewise, if the general fiber is a septic curve, by Theorem~\ref{thm:possibledualgraphs} and Lemma~\ref{degree7noindex5surface} there are only 7 possible intersection graphs.

\begin{remark}\label{rem:semistablereduction}
Without loss of generality we make make several assumptions:
\begin{enumerate}
\item there is a birational map $\Dcy\dra D$ (as the general fibers are isomorphic),
\item by semistable reduction we may assume there is a resolution of singularities $\Dss$ of $\Dcy$ such that $\Dsso$ is reduced and nodal,
\item the map $\Dss\ra \Dnorm$ is an isomorphism away from $\Dsso$ and the $\delta$-invariants of $\Dnormo$ are determined by Lemma~\ref{deltabranches},
\item the intersection graph $\Gamma(\Dnormo)$ is a tree, and
\item the map $\Dss\ra D$ is regular and only contracts trees of rational curves (as $D$ is the relative minimal model).
\end{enumerate}
\end{remark}

The following lemma implies that if $\Dcyo$ is reduced, then there is a unique singularity.

\begin{lemma}\label{lem:uniquesingularpt}
If $\Dcyo$ is reduced and singular at some point $P \in \Dcyo$, then $P$ is the unique singular point of $\Dcyo$ and all components of $\Dcyo$ are rational.
\end{lemma}

\begin{proof}
Let $d$ be the degree of the general fiber of $\Dcy$.  By Theorem~\ref{thm:possibledualgraphs} and Lemma~\ref{lem:Dnormsatisfies(Hi)}, all components of $\Dcyo$ must intersect at a unique point $Q$ and there are no singularities on the components with multiple branches. So by Remark~\ref{rem:semistablereduction}(3) and Lemma~\ref{deltabranches}(2) the $\delta$-invariant of any point $P\in \Dcyo\setminus Q$ is either $0$ or $g$ (the arithmetic genus of $\Dcyo$). If there are multiple components, then for any component $C\subset \Dcyo$ and any point $P\in C\setminus Q$, $\delta_P(C)$ is bounded by the arithmetic genus of $C$ which is less than $g$, so $\delta_P(C)=0$ and $\Dcyo$ is smooth at $P$.

Similarly, if $\Dcyo$ consists only of one component, then as every point is unibranch (by Theorem~\ref{thm:possibledualgraphs}) by Remark~\ref{rem:semistablereduction}(3) and Lemma~\ref{deltabranches}(2) there is a unique point where the $\delta$-invariant is not 0, so there is a unique singularity.

If $\Dcyo$ is reduced and singular, then no component can have geometric genus $g$, so the strict transform of each component must be contracted in the map $\Dss\ra D$ (in Remark~\ref{rem:semistablereduction}(5)). Thus every component is rational.
\end{proof}

The proofs of Theorems \ref{thmD:deg5} and \ref{thmE:deg7} are casework: ruling out all possible singular limits. All but three cases can be ruled out by straightforward calculations of the log-canonical threshhold. The three possible cases for $\Dcyo$ that cause the most difficulty all occur for degree 7 limits:
\begin{enumerate}
\item a reduced irreducible rational degree 7 plane curve with a single unibranch singularity,
\item a reduced irreducible rational degree 14 curve with a single unibranch singularity in $\PP(1,1,4)$, and
\item the union of a doubled plane cubic and an inflection line.
\end{enumerate}
These special cases use previously known classification results on unibranch rational curves of low degree, which are recorded below. For more background references see \cite[Ch. 8]{PlaneAlgCurves}, \cite[Ch. 2]{TKThesis}, and \cite[Ch. 2]{MoeThesis}.

Throughout, we say a singular point of a curve $p\in C$ is a \textit{cusp} if $C$ is unibranch at $p$. Let $p \in C \subset  X$ be a cuspidal point $p$ on the curve $C$ on a smooth surface $X$.  In an analytic neighborhood of $p$, the curve $C$ can be written parametrically as
\[
(x,y) = (t^a, c_1t^{b_1} + c_2t^{b_2} + \dots)
\]
with $1 < a < b_1 < b_2 < \dots\in \ZZ$ such that $a$ does not divide $b_1$, $\gcd(a, b_1, b_2, \dots) = 1$, and $c_i \ne 0$ for all $i$. From the parametrization of a cuspidal singularity it is possible to read off many invariants of the singularity. For example, the multiplicity of the cusp $C$ is $a$ and the log canonical threshold of the pair $(X,C)$ near $p$ is $(1/a) + (1/b_1)$.

\begin{definition}\label{defn:multiplicitysequence}
One invariant of the cusp is the \textbf{multiplicity sequence}; the sequence encoding the multiplicity of the exceptional divisors in the minimal resolution of the cusp.  Let $(X_0, C_0) := (X,C)$ and let
\[
\pi_i: (X_i,C_i) \to (X_{i-1}, C_{i-1})
\]
be the blow up of the singular point of $C_{i-1}$ with exceptional divisor $E_i$. Set $C_i = (\pi_i)^{-1}_* C_{i-1}$.  Let $\pi = \pi_n \circ \pi_{n-1} \dots \circ \pi_1$ be the minimal embedded resolution of the cusp $p$.  The \textit{multiplicity sequence of the cusp} $p \in C$ is the sequence $\overline{m}_p := (m_1, m_2, \dots, m_n)$, where $m_i$ is the multiplicity of the exceptional divisor $E_i$ in $\pi_i^*(C_{i-1})$.  This satisfies the inequalities $m_1 \ge m_2 \ge \dots \ge m_n = 1$.  For simplicity, if $C_{i-1}$ is smooth then $m_i = 1$ (and hence $C_j$ is smooth for all $i \le j \le n$), we omit all the multiplicities $m_j =1$ for $j\ge i$.  
\end{definition}

The $\delta$-invariant (Definition \ref{defn:delta}) of a cusp singularity can be read off from its multiplicity sequence: 
\begin{equation}\label{deltamultiplicity}
    \delta_p = \sum_{i = 0}^n \frac{m_i(m_i-1)}{2} 
\end{equation}

\begin{definition}\label{defn:newtonpairs}
Another invariant of the cusp is the collection of \textbf{Newton pairs} that parameterize the cusp.  Define $g_i := \gcd(a,b_1,b_2, \dots ,b_i)$.  Then, there is a finite sequence $i_1 < i_2 < \dots < i_k$ at which $g_i$ decreases, i.e. $i_1 = 1$,
\[
g_{i_1} = \dots = g_{i_2 - 1} > g_{i_2} = \dots =g_{i_3-1} > g_{i_3} = \dots > g_{i_k} = 1.
\]
Define $i_0 = 0$, $b_0 = 0$, and $g_{i_0} = a$. For $1 \le j \le k+1$, let $M_j = g_{i_{j-1}}$ and for $1 \le j \le k$, let $N_j = b_{i_j} - b_{i_{j-1}}$.  The $k$ \textit{Newton pairs} of the cusp are the $k$ pairs
\[
(m_j, n_j) = \left(\frac{M_j}{M_{j+1}}, \frac{N_j}{M_{j+1}} \right)\text{ for } 1 \le j \le k .
\]
\end{definition}

We can write $M_j = m_j m_{j+1} \dots m_k$ and $N_j = n_j m_{j+1} \dots m_k$. The $\delta$-invariant of a cusp singularity can also be expressed in terms of the $M_i$ \cite[2.1.1]{TKThesis}:
\begin{equation}\label{delta}
\delta_p = \frac{1}{2} \left((M_1-1)(N_1-1) + \sum_{j=2}^k(M_j-1)N_j \right).
\end{equation}

\begin{remark}\label{rmk:boundonnewtonpairs}
By construction, $M_{j} \ge 2M_{j+1}$, so $M_j \ge 2^{k-j}M_k$, and similarly, $N_j\ge 2^{k-j} N_k$.  Because $M_k \ge 2$ and $N_k \ge 1$, $M_j \ge 2^{k-j+1}$ and $N_j \ge 2^{k-j}$.  These inequalities relate the number of Newton pairs $k$ to the $\delta$-invariant.  

Further observe that $N_1 = b > a = M_1$ by construction, so we have the bound 
\begin{align*}
    2\delta_p &= (M_1-1)(N_1-1) + \sum_{2\le j\le k}(M_j-1)N_j  \\ 
    &\ge (M_1 - 1)(N_1 - 1) \\
    &\ge (M_1 -1)M_1 \\
    & \ge (2^k - 1)2^k. 
\end{align*}
\noindent In particular, if $k \ge 2$, the first inequality is strict, so for $k = 2$ we have $2 \delta_p > 12$, so $\delta_p \ge 7$.  If $k = 3$, we obtain $2\delta_p > 56$ so $\delta_p \ge 29$.  These bounds will be used to rule out certain cuspidal curves below. 
\end{remark}

From the Newton pairs and multiplicity sequences, it is possible to list all \textit{unicuspidal} rational curves of low degree, i.e. curves $C$ of degree $d$ with one isolated unibranch singularity at $p \in C$ such that $\delta_p = g(d)$, the genus of a degree $d$ plane curve (and no other singularities).  Note that not all numerical solutions to equations such (\ref{deltamultiplicity}) or (\ref{delta}) can actually occur as multiplicities or Newton pairs of plane curves.  The following table lists the rational cuspidal curves of degrees $3$ through $6$ with a single cusp.  These results are for $d = 3, 4, 5$ are derived in \cite[Tables 3.1, 3.2, 6.1]{MoeThesis}.  Alternatively, because $\delta_p  < 7$ for $d \le 5$, by Remark \ref{rmk:boundonnewtonpairs}, the cusp is parameterized by a single Newton pair, so the classification also appears in \cite[Thm. 1.1]{Unicuspidalcurves}.  For $d = 6$, there are at most two Newton pairs parameterizing the cusp, so one can obtain explicit equations from \cite[Thm. 1.1]{Unicuspidalcurves} and \cite[Thm. 1.1]{TKThesis}.  The only case with two Newton pairs is given in line 7 in the table below and corresponds to \cite[Thm. 1.1(4)]{TKThesis} and the local equation and log canonical threshold can be worked out by hand.  The classification of rational unicuspidal plane curves of degree $\le 6$ is listed in Table \ref{tableofcusps}.

\begin{center}
\begin{table}[H] 
\caption{Rational unicuspidal plane curves of degree $\le 6$.}
\label{tableofcusps}
\scalebox{.9}{
\begin{tabular}{c|c|c|c|c|c}
    Degree & $\begin{array}{c}\text{Parameterization} \end{array}$& $\begin{array}{c}\text{Local equation}\\\text{of cusp} \end{array}$ & $\begin{array}{c}\text{Multiplicity}\\\text{sequence} \end{array}$ & $\begin{array}{c}\text{Newton}\\\text{pairs} \end{array}$ & $\begin{array}{c}\text{Log canonical}\\\text{threshold} \end{array}$  \\\hline
    3 & $ (x,y) =(t^2,t^3)$ & $y^2 = x^3$ & $(2)$ & $(2,3)$ & $5/6$ \\\hline 
    4 & $(x,y) =(t^2,t^7)$ & $y^2 = x^7$ & $(2,2,2)$ & $(2,7)$ & $9/14$ \\\hline 
    4 & $(x,y) =(t^3,t^4)$ & $y^3 = x^4$ & $(3)$ & $(3,4)$ &  $7/12$ \\\hline 
    5 & $(x,y) =(t^2,t^{13})$ & $y^2 = x^{13}$ & $(2,2,2,2,2,2)$ & $(2,13)$ & $15/26$ \\\hline 
    5 & $(x,y) =(t^4,t^5)$ & $y^4 = x^5$ & $(4)$ & $(4,5)$ & $9/20$ \\\hline
    6 & $(x,y) =(t^3,t^{11})$ & $y^3 = x^{11}$ & $(3,3,3,2)$ & $(3,11)$ & $14/33$ \\\hline
    6 & $(x,y) =(t^4,t^6+t^{11})$ & $\begin{array}{c} y^4 = 2x^3y^2 - x^6 \\ + 4x^7y + x^{11} \end{array}$ & $(4,2,2,2,2)$ & $(2,3),(2,5)$ & $5/12$  \\\hline
    6 & $(x,y) =(t^5,t^6)$ & $y^5 = x^6$ & $(5)$ & $(5,6)$ & $11/30$ \\\hline
\end{tabular}
}
\end{table}
\end{center}

\vspace{-.3in}

For higher degree curves, it is necessary to appeal to stronger invariants to understand the possible singularities: 

\begin{definition}\label{defn:semigroup}
The \textit{semigroup} of a cuspidal singularity $p\in C$, denoted $W_p\subset \NN$, is the set of local intersection multiplicities of $C$ with other curves at $p$, i.e.
\[
W_p:= \{ \dim_{\mathbb{C}} \Oc_{C,p} /(f) \mid f \in \Oc_{C,p}, f \ne 0 \} \subset \NN.
\]
\end{definition}

$W_p$ has a set of minimal generators $\{0,w_1, w_2, \dots , w_{k+1} \}$. The $w_i$ can be expressed in terms of the Newton pairs of the cusp:
\[
w_1 = M_1, \quad w_2 = N_1, \quad w_j = m_{j-2}w_{j-1} + N_{j-1} \quad 3 \le j \le k+1.
\]

\begin{example}
In the case of a cusp $p\in C$ with one Newton pair $(a,b)$, the curve can be analytically locally parametrized by $(x,y)=(t^a,t^b)$ with $\gcd(a,b)=1$. $C$ has analytic local equation: $x^b=y^a$, log canonical threshhold: $(1/a)+(1/b)$, $\delta$-invariant: $(a-1)(b-1)/2$, and semigroup $W_p=\langle 0,a,b\rangle.$
\end{example}

If $p\in C$ is a cusp with semigroup $W_p$, define
\[
R_p(k) := \#\{ W_p \cap [0,k) \}
\]
to be the counting function of elements in $W_p$ between $0$ and $k-1$. The counting function satisfies interesting properties. In particular, if a curve $C$ is a cuspidal $d$ plane curve, these counting functions must satisfy particular constraints related to the degree $d$ of the curve.  The multiplicity, Newton pairs, and intervals in the semigroup have been widely used to study plane curves, e.g. \cite{MatsuokaSakai,Orevkov,Unicuspidalcurves,BorLiv}.

In \cite[Thm. 6.5, Rem. 6.6]{BorLiv}, Borodzik and Livingston prove the following strong result on existence of such curves: for a rational cuspidal curve of degree $d$ with $n$ cusps $p_1, \dots, p_n$ and counting functions $R_{p_1}, \dots, R_{p_n}$, then for any $j \in \{ -1, \dots, d-2\}$, 

\begin{equation}\label{countingfunction}
    \min_{\substack{k_1, \dots, k_n \in \mathbb{Z}; \\ k_1 + \dots + k_n = jd + 1}} ( R_{p_1}(k_1) + \dots + R_{p_n}(k_n)) = \frac{(j+1)(j+2)}{2}.
\end{equation}

We conclude with a useful computation of log canonical threshholds.

\begin{lemma}\label{lctcomputationfortwocomps}
Let $R_1,...,R_k\subset S$ be curves in a smooth surface $S$ such that $P\in S$ is the unique point in $R_1\cap R_i$ (or work with the local intersection numbers). Assume
\[
\length_P(R_1\cap R_i) = \ell_i.
\]
Assume $R_1$ has a singularity at $P$ with analytic local equation $x^a = y^b$ with $a$ and $b$ coprime. Assume that analytically locally at $P$, the curve $R_i$ is unibranch. If we set the weight of $x$ equal to $b$ and the weight of $y$ equal to $a$, then the weight $w_i$ of the local analytic equation of $R_i$ satisfies:
\[
w_i\ge \min\{ab,\ell_i\}.
\]
It follows that
\[
\lct(S,c_1R_1+...+c_kR_k) \le \frac{a+b}{c_1ab+c_2\min\{ab,\ell_1\}+\cdots+c_k\min\{ab,\ell_k\}}.
\]
A useful special case is the case of two curves $R_1$ and $R_2$ where $R_1$ is smooth at $P$, $R_2$ is unibranch at $P$ and they meet to length $\ell$. In this case, analytic locally we can write $R_1$ as $x=y^\ell$, and we have
\[ 
\lct(S, c_1R_1 + c_2 R_2) \le \frac{1 + \ell}{\ell(c_1 + c_2)}.
\]
\end{lemma}

\begin{proof}
Let $\mu\cl\St\ra S$ be the $(b,a)$ weighted blow-up of $(x,y)$ with exceptional divisor $E$. Let $\tilde{R_1}$ and $\tilde{R_i}$ be the strict transforms of $R_1$ and $R_i$. Then
\begin{center}
$\ell_i  = \tilde{R}_1 \cdot \pi^*R_i = \tilde{R}_1 \cdot (\tilde{R}_i +w_iE) = \tilde{R}_1 \cdot \tilde{R}_i + w_i,$ and\\
$\ell_i  = \tilde{R}_i \cdot \pi^*R_1 = \tilde{R}_i \cdot (\tilde{R}_1 + (ab)E) = \tilde{R}_1 \cdot \tilde{R}_i + ab(\tilde{R}_2\cdot E).$
\end{center}
Thus $\tilde{R}_i\cdot E = w_i/ab.$ If the weight $w_i<ab$ then this intersection is fractional which implies that $\tilde{R}_i$ meets $E$ at a singularity. But, the $(b,a)$ weighted blow-up of $x^a=y^b$ has the property that it is a resolution of $R_1$ and $\tilde{R}_1$ meets $E$ at exactly one point that is smooth. As $R_1$ and $R_i$ are both unibranch at $P$, $\tilde{R}_1$ and $\tilde{R}_i$ do not intersect over $P$, so $\tilde{R}_1 \cdot \tilde{R}_i = 0$, i.e. $\ell_i = w_i$.  This shows that either $w_i \ge ab$ or $w_i = \ell_i$, and hence $w \ge \min\{ab, \ell\}$. To finish, note that 
\[
K_{\St} = \mu^*K_S + (b+a-1)E,\text{ and }\mu^*\left(\sum c_i R_i\right) = \left(\sum c_i\tilde{R}_i\right)+ (c_1ab+c_2w_i+\dots+c_kw_k)E.
\]
Therefore, the log canonical threshold of the pair satisfies \[ \lct(S, \sum c_i R_i) \le \frac{a+b}{c_1ab + c_2 w_2+\dots+c_iw_i} \le  \frac{a+b}{c_1ab+c_2\min\{ab,\ell_1\}+\cdots+c_k\min\{ab,\ell_k\}}.\qedhere\]
\end{proof}

\section{On reduced limits of plane curves}\label{sec:reducedlimits}

In this section we work in the setting of \S\ref{sec:curvebackground} and study the possible reduced curves that can appear as $\Dcyo$. For degrees 5 and 7 we prove that if $\Dcyo$ is reduced and singular, then it is fact irreducible, rational, and has a unique singular point which is unibranch; i.e. $\Dcyo$ is \textit{unicuspidal}. By proving that the log canonical threshold of rational unicuspidal plane curves of degree $d = 5$ or $d =7$ is strictly smaller that $\frac{3}{d}$, we prove that these cannot appear as $\Dcyo$. Therefore, if $\Dcyo$ is reduced, it must be smooth.  This proves Theorem \ref{thmD:deg5} because $\Dcyo$ is always reduced when $d = 5$, and reduces Theorem~\ref{thmE:deg7} to the case $\Dcyo$ is nonreduced.

\subsection{On Hacking's Calabi-Yau limits of quintic plane curves that are reduced}
In this subsection we prove Theorem \ref{thmD:deg5}.  All components of $\Dcyo$ are reduced, so we show that the only reduced limits of families of quintic curves are smooth curves $\Dcyo \subset \PP^2$ or $\Dcyo \subset M(5)$.  For limits in $M(5)$, the proof of Theorem~\ref{smoothlimits1} shows that the curves $\Dcyo$ are all hyperelliptic.

First we show:

\begin{proposition}\label{reducedimpliessmoothdeg5}
If $D$ is a smooth family of curves such that the general fiber is a plane quintic then Hacking's Calabi-Yau limit $\Dcyo$ is smooth.
\end{proposition}

\begin{proof}
Suppose for contradiction that $\Do$ is a smooth limit of a family of plane quintics, and $\Dcyo$ is singular.  By Lemma \ref{lem:uniquesingularpt}, $\Dcyo$ has a unique singularity at $P$, all components are rational, and by Theorem~\ref{thm:possibledualgraphs} and Lemma~\ref{lem:Dnormsatisfies(Hi)} there are at most 3 components with specified intersection graph and all components are reduced.

In Lemma~\ref{p2quintics} we show there are no possible singular limits $\Dcyo$ in $\PP^2$. In Lemma~\ref{p114quintics} we show there are no possible singular limits $\Dcyo$ in $\PP(1,1,4)$. In Lemma~\ref{p1425quintics} we show there are no possible singular limits $\Dcyo$ in $\PP(1,4,25)$. In Lemma~\ref{m5quintics} we show there are no possible singular limits $\Dcyo$ in $M(5)$. Therefore by Theorem~\ref{degree5surfaces}, the result follows.
\end{proof}

Theorem \ref{thmD:deg5} follows easily from the proposition.

\begin{proof}[Proof of Theorem \ref{thmD:deg5}]
By Proposition~\ref{reducedimpliessmoothdeg5}, Hacking's Calabi-Yau limit $\Dcyo\subset \Xcyo$ is smooth.  By Proposition~\ref{degree5surfaces}, $\Xcyo$ is either $\PP^2$, $\PP(1,1,4)$, $\PP(1,4,25)$, or $M(5)$. By Theorem~\ref{thm:classpicofmanettisurfaces}, limits of quintic curves on $\PP(1,1,4)$ and $\PP(1,4,25)$ are not Cartier and by degree considerations cannot be smooth. Therefore, $\Xcyo = \PP^2$ or $M(5)$ and $\Dcyo$ is a Cartier divisor on $\Xcyo$. If $\Xcyo = M(5)$, then by the computation in the proof of Theorem \ref{smoothlimits1} $\Dcyo$ is hyperelliptic. 
\end{proof}

For the following lemmas, we assume the following about $\Dcyo$ as in the proof of Proposition~\ref{reducedimpliessmoothdeg5}.

\begin{enumerate}
\item[($\ast$)] \textit{$\Dcyo$ has at most three components, all of which are rational and unibranch, there is a unique singular point $P\in \Dcyo$ of multiplicity at most $3$, and $\lct(\PP^2,\Dcyo)\ge 3/5$.}
\end{enumerate}

\begin{lemma}\label{p2quintics}
    There is no singular limit $\Dcyo$ satisfying $(\ast)$ in $\PP^2$.
\end{lemma}

\begin{proof}
From $(\ast)$, there are five cases to rule out:
\[
\begin{array}[t]{c|c|c|c|c|c}
&\text{Case 1}&\text{Case 2}&\text{Case 3}&\text{Case 4}&\text{Case 5}\\
\hline
\text{Components}&3&3&2&2&1\\
\hline
\text{Degrees}&3+1+1&2+2+1&4+1&3+2&5
\end{array}
\]
\begin{enumerate}[leftmargin=4.5\parindent]
\item[(Case 1)] The cubic is cuspidal at $P$, so the multiplicity is 4, a contradiction to $(\ast)$.
\item[(Case 2)] Two conics and a line meeting at one point has $\lct(\PP^2,\Dcyo) = 1/2$ (Lemma~\ref{lctcomputationfortwocomps}), a contradiction to $(\ast)$.
\item[(Case 3)] There are two rational unicuspidal quartics (see Table~\ref{tableofcusps}: one has multiplicity 3 at the singular point, so the union with the line has multiplicity 4, contradicting ($\ast$). The other has local equation $x^2+y^7$ and meets the line to length 4, so by Lemma~\ref{lctcomputationfortwocomps}, $\lct(\PP^2,\Dcyo)\le 1/2<3/5$, a contradiction to $(\ast)$.
\item[(Case 4)] The cubic is cuspidal at $P$ and meets the conic to order 6, so by Lemma~\ref{lctcomputationfortwocomps}, $\lct(\PP^2,\Dcyo)\le 5/12<3/5,$ a contradiction to $(\ast)$.
\item[(Case 5)] There are two rational unicuspidal quintics (Table \ref{tableofcusps}): one has a multiplicity 4 point, a contradiction. The other has a singularity with analytic equation $y^2=x^{13}$, which has log canonical threshold $15/26<3/5$, a contradiction to $(\ast)$.\qedhere
\end{enumerate}
\end{proof}

\begin{lemma}\label{p114quintics}
    There is no singular limit $\Dcyo$ satisfying $(\ast)$ in $\PP(1,1,4)$.
\end{lemma}

\begin{proof}
    Assuming that such a limit $\Dcyo\subset \PP(1,1,4)$ existed, it would have degree 10. By Lemma~\ref{remainders}, $\Dcyo$ is singular at the vertex, so by ($\ast$) every component of $\Dcyo$ passes through this point.
    
    If $\Dcyo$ has three components, then by the multiplicity bound in ($\ast$) each component passes through the vertex with multiplicity 1, but $10 = 2 \ne 3 \pmod 4$, so this impossible.
    
    If $\Dcyo$ has two components, then by the multiplicity bound in ($\ast$), Lemma~\ref{remainders}, and because $10 = 2 \pmod 4$, both components must pass through the vertex with multiplicity 1. Therefore, by ($\star$) both components are smooth and rational. The possible degrees of such a configuration are $9+1$ or $5+5$. Degrees $9+1$ is impossible as a smooth degree 9 curve in $\PP(1,1,4)$ is not rational. The second case is impossible as the vertex can be the only point of intersection of the two curves (which are each smooth at the vertex); blowing up this point yields two smooth curves on $\FF_4$ meeting to order 6, which by Lemma~\ref{lctcomputationfortwocomps} has log canonical threshold at most $7/12$, which is too small.

    Finally, suppose the curve $\Dcyo$ is irreducible with a unicuspidal singularity at the vertex.  By blowing up the vertex $\pi: \FF_4 \to \PP(1,1,4)$, we produce a unicuspidal curve on $\FF_4$.  By ($\ast$), we know $\lct(\PP(1,1,4), \Dcyo)\ge 3/5$.  Let $E$ be the exceptional divisor of $\pi$ and consider the pair of equations 
    \[
        \pi^*(K_{\PP(1,1,4)}) = K_{\FF_4} + \frac{1}{2} E,\text{ and }\quad \pi^*(\Dcyo) = \widetilde{\Dcyo} + \frac{a}{4} E
    \]
    where $a \in \mathbb{Z}^+$ and $\widetilde{\Dcyo}$ is the strict transform of $\Dcyo$.
    
    We analyze the rational unicuspidal curve $\widetilde{\Dcyo}$. By the log canonical threshold assumption, \[\frac{1}{2} + \frac{3a}{20} \le 1,\] so $a \le 3$.  Intersecting the second equation with a fiber $\ell$ of the ruled surface $\FF_4$ (whose image is a section of $\Oc(1)$ on $\PP(1,1,4)$) gives \[ \frac{10}{4} = \widetilde{\Dcyo} \cdot \ell + \frac{a}{4} \] and as $\widetilde{\Dcyo} \cdot \ell \in \mathbb{Z}$, this proves that $a = 2$.  The arithmetic genus of $\widetilde{\Dcyo}$ is computed by: 
    \[ 2g_a(\widetilde{\Dcyo}) - 2 = (K_{\FF_4} + \widetilde{\Dcyo})\cdot \widetilde{\Dcyo} = (K_{\PP(1,1,4)} + \Dcyo)\cdot \Dcyo - E \cdot \widetilde{\Dcyo} = 10 - 2 = 8  \] so the arithmetic genus is 5, which is the same as the $\delta$-invariant of the cusp in $\widetilde{\Dcyo}$. By Remark~\ref{rmk:boundonnewtonpairs}, the cusp must be parameterized by a single Newton pair $(M_1, N_1) = (a,b)$ with
    \[ 10 = (M_1 - 1)(N_1 - 1) = (a-1)(b-1).\]

    The possible values of $(a,b)$ are $(a,b) = (2,11)$ or $(3,6)$.  In the first case, the log canonical threshold is less than $\frac{3}{5}$, which is too small, and in the second case, the curve is not unibranch, so both give a contraction to $(\ast)$.  Therefore, no such curve exists on $\PP(1,1,4)$.
\end{proof}

\begin{lemma}\label{p1425quintics}
    There is no singular limit $\Dcyo$ satisfying $(\ast)$ in $\PP(1,4,25)$.
\end{lemma}

\begin{proof}
    This follows from the previous lemma: in this case all of the computations in Lemma~\ref{p114quintics} can be done locally around the $\frac{1}{4}(1,1)$ singularity and yield the same contradictions.  
\end{proof}

\begin{lemma}\label{m5quintics}
    There is no singular limit $\Dcyo$ satisfying $(\ast)$ in $M(5)$.
\end{lemma}

\begin{proof}
The surface $M(5)$ is the partial smoothing of the $\frac{1}{4}(1,1)$ singularity on $\PP(1,4,25)$, and can be realized as a degree 26 surface in the weighted projective space $\PP(1,2,13,25)$.  The curve $\Dcyo$ is a complete intersection of the degree 26 surface and a degree 25 surface.  Denote by $[x:y:z:w]$ the weighted coordinates on $\PP(1,2,13,25)$.  Up to change of coordinates, we can realize this pair $(M(5), \Dcyo)$ as $((xw = f_{26}(x,y,z)), (w = 0))$ in $\PP(1,2,13,25)$, where $f_{26}$ is a generic polynomial of degree 26 in $x,y,z$.  Therefore, the curve $\Dcyo$ can be expressed as $(f_{26}(x,y,z) = 0) \subset \PP(1,2,13)$. (See also Example~\ref{ex:partialsmoothingconstruction} and Example~\ref{ex:quintics}.)

If the curve has three components of degrees $d_1$, $d_2$, and $d_3$ with $d_1+d_2+d_3=26$ then by $(\ast)$, they must all be smooth and intersect at a unique point.  Because at most one $d_i$ can be a multiple of 13, the remaining two curves must go through the $\frac{1}{13}(1,2)$ singular point.  Because they are smooth at this point, we claim that each $d_i$ must be congruent to $1 \pmod{13}$.  Indeed, writing the coordinates on $\PP(1,2,13)$ as $[x:y:z]$, a curve of degree $d_i$ has the form 
    \[ \left( \sum_{i = 0}^{\lfloor d_i/13 \rfloor} z^i f_{d_i-13i}(x,y) = 0 \right) \]
where $f_{d_i - 13i}(x,y)$ is a polynomial of degree $d_i - 13i$.  If $d_i \ne 1 \pmod{13}$, then in a neighborhood of the point $[0:0:1]$, this vanishes to order at least two, hence is not smooth.  But, it is impossible that $d_1 + d_2 + d_3 = 26$ and each $d_i \equiv 1 \pmod{13}$. 

If the curve has two components of degrees $d_1$ and $d_2$, by $(\ast)$ it can pass through at most one of the singular points of $\PP(1,2,13)$, so both degrees must either be $0 \pmod 2$ or $0 \pmod{13}$.  Because $d_1 + d_2 = 26$, these are mutually exclusive, and for $n \in \{2,13\}$,  $d_1 \ne 0 \pmod{n} \iff d_2 \ne 0 \pmod{n}$.  Therefore, both components must pass through exactly one of the singular points.  Suppose the curves contain the $\frac{1}{13}(1,2)$ singularity and $d_i \equiv 0 \pmod{2}$.  Because the curve has multiplicity at most three at the singularity, if $r_i$ is the remainder of $d_i \pmod{13}$ then $r_1+r_2 \le 3.$ However, $d_1 + d_2 = 26$, and this is impossible.

Now suppose that the curves contain the $\frac{1}{2}(1,1)$ singularity and $d_i \equiv 0 \pmod{13}$.  Because $d_1 + d_2 = 26$, this implies that $d_1 = d_2 = 13$.  By $(\ast)$, these two unibranch curves meet only at the $\frac{1}{2}(1,1)$ singularity.  Blowing up this singular point of the surface yields two curves intersecting at one point to order at least 5, contradicting that the log canonical threshold is at most $\frac{3}{5}$.  

Finally, assume that the curve has only one component.  Because the curve has degree $26$ in $\PP(1,2,13)$ and has only unibranch singularities, it must avoid the singularities of the surface.  Suppose there is a cusp in the smooth locus of the surface.  The curve has arithmetic genus 6, so by Remark \ref{rmk:boundonnewtonpairs} it is parameterized by a single Newton pair $(M_1, N_1) = (a,b)$ such that 
    \[ 12 = (M_1 - 1)(N_1 - 1) = (a-1)(b-1).\]
The only solutions are $(a,b) = (2,13)$, $(3,7)$, or $(4,5)$, and each of these has log canonical threshold smaller than $\frac{3}{5}$, contradicting $(\ast)$.
\end{proof}

\subsection{On Hacking's Calabi-Yau limits of septic plane curves that are reduced}
In this section we show that the only reduced limits of degree 7 curves are smooth curves $\Dcyo\subset \PP^2$. Throughout, we assume for contradiction that $\Dcyo$ is reduced but singular. By Lemma~\ref{lem:uniquesingularpt}, Theorem~\ref{thm:possibledualgraphs}, and Lemma~\ref{lem:Dnormsatisfies(Hi)} we make the following assumptions about $\Dcyo$:

\begin{enumerate}
\item[($\diamond$)] \textit{$\Dcyo$ has at most four components, all of which are rational and unibranch, there is a unique singular point $P\in \Dcyo$ of multiplicity at most $4$, and $\lct(\PP^2,\Dcyo)\ge 3/7$.}
\end{enumerate}

\begin{proposition}\label{reducedimpliessmooth}
If $D$ is a smooth family of curves such that the general fiber is a plane septic and Hacking's Calabi-Yau limit $\Dcyo$ is reduced then $\Dcyo$ is smooth.
\end{proposition}

\begin{proof}
For contradiction we assume $\Dcyo$ is singular, so we may assume ($\diamond$). Except for the case of a single component, these possibilities can all be ruled out by computing that their log canonical thresholds are too small. These computations are carried out in the following lemmas as indicated in the table.
\[
\begin{array}{cc}
\text{Curves in $\PP^2$} & \text{Curves in $\PP(1,1,4)$}\\
\begin{array}[t]{c|c|l}
\begin{array}{c}\text{Components}\end{array}&\begin{array}{c}\text{Degrees}\end{array}&\begin{array}{c}\text{Proof}\end{array}\\\hline
4&2+2+2+1&\text{Lem.~\ref{fourredcoms}}\\\hline
3&3+2+2&\text{Lem.~\ref{threeredcoms}}\\\hline
3&4+2+1&\text{Lem.~\ref{threeredcoms}}\\\hline
3&5+1+1&\text{Lem.~\ref{threeredcoms}}\\\hline
2&6+1&\text{Lem.~\ref{tworedcoms}}\\\hline
2&5+2&\text{Lem.~\ref{tworedcoms}}\\\hline
2&4+3&\text{Lem.~\ref{tworedcoms}}\\\hline
1&7&\text{Lem.~\ref{oneredcom}}\\\hline
\end{array}& \begin{array}[t]{c|c}
\begin{array}{c}\text{Components}\end{array}&\begin{array}{c}\text{Proof}\end{array}\\\hline
\text{multiple}&\text{Lem.~\ref{multiplep112}}\\\hline
\text{one}&\text{Prop.~\ref{oneredcomp114}}\\\hline
\end{array}
\end{array}
\]
\end{proof}

\begin{lemma}\label{fourredcoms}
There is no reduced limit $\Dcyo\subset \PP^2$ satisfying $(\diamond)$ with 4 components.
\end{lemma}

\begin{proof}
A line and 3 conics meeting at 1 point has $\lct(\PP^2,\Dcyo)=4/11$, a contradiction.
\end{proof}

\begin{lemma}\label{threeredcoms}
There is no reduced limit $\Dcyo\subset \PP^2$ satisfying $(\diamond)$ with 3 components.
\end{lemma}

\begin{proof}
Suppose $\Dcyo=R_1+R_2+R_3$ is such a limit satisfying $(\diamond)$ and assume that $R_1$ is the curve of degree $\ge 3$ -- so $R_1$ is cuspidal. By multiplicity considerations, the singularity of $R_1$ must be a double point -- so analytically locally has equation $y^2=x^{2\delta+1}$ ($\delta$ is the $\delta$-invariant).

If the degrees are 3+2+2 then the cusp has local equation $y^2=x^3$ and the conics must meet the cubic to length 4. By Lemma~\ref{lctcomputationfortwocomps} we have
\[
\lct(\PP^2,\Dcyo)\le \frac{5}{6+6+6}<3/7,
\]
a contradiction. If the degrees are 4+2+1 then the quartic cusp has local equation $y^2=x^7$. So by Lemma~\ref{lctcomputationfortwocomps},
\[
\lct(\PP^2,\Dcyo)\le \frac{9}{14+8+4}<3/7.
\]
If the degrees are 5+1+1, then the quintic has local equation $y^2=x^{13}$. So by Lemma~\ref{lctcomputationfortwocomps},
\[
\lct(\PP^2,\Dcyo)\le \frac{15}{26+5+5}< 3/7.
\]
\end{proof}

\begin{lemma}\label{tworedcoms}
There is no reduced limit $\Dcyo=R_1+R_2\subset \PP^2$ with two components satisfying $(\diamond)$.
\end{lemma}

\begin{proof}
Suppose $R_1+R_2$ has degrees $6+1$. By Table~\ref{tableofcusps}, there are three rational unicuspidal sextic curves.  Two of them have multiplicity at least 4, so the multiplicity of the intersection point of $R_1+R_2$ is too large.  The remaining case has log canonical threshold $\frac{14}{33}$ which is smaller than $\frac{3}{7}$, so the union $R_1 + R_2$ is too singular.  Therefore, degrees 6+1 for $R_1$ and $R_2$ are impossible.

Next suppose $R_1+R_2$ has degrees $5+2$.  By Table~\ref{tableofcusps}, there are two rational unicuspidal quintics: one has a multiplicity 4 point, and the other has a singularity with analytic equation $y^2=x^{13}$. In the first case the multiplicity of $R_1+R_2$ at the intersection point is $\ge 5$ which is too singular. In the second case, in these analytic coordinates we consider $y$ with weight 13 and $x$ with weight 2. As the intersection $R_1\cap R_2$ has length 10, Lemma \ref{lctcomputationfortwocomps} implies that $\lct(\PP^2,\Dcyo)\le 15/(26+10) = 5/12$ which is too small, thus degrees 5 and 2 are impossible.

Finally, suppose $R_1+R_2$ has degrees $4+3$.  There are two rational unicuspidal quartics (Table \ref{tableofcusps}): one has a multiplicity 3 point, and the other has a singularity with analytic equation $y^2+x^7=0$. In the case of the multiplicity 3 singularity, $R_2$ must be a cubic with an ordinary cusp at $P$. In this case, $\Dcyo$ has multiplicity 5 at that point, which is too great. In the case that $R_1$ has a singularity of the form $y^2+x^7=0$, then Lemma \ref{lctcomputationfortwocomps} implies that $\lct(\PP^2,\Dcyo) \le 9/26 <3/7$ so is too singular.  Therefore, degrees 4 and 3 are impossible.
\end{proof}

\begin{lemma}\label{multiplep112}
There is no reduced limit $\Dcyo\subset \PP(1,1,4)$ with multiple components.
\end{lemma}

\begin{proof}
By Lemma \ref{lem:uniquesingularpt} and Lemma~\ref{remainders}, the unique singular point of $\Dcyo$ must be at the vertex, which is the only place the components of $\Dcyo$ can intersect.  Therefore, by the log canonical threshold assumption, there can be at most four components of $\Dcyo$.  If there are exactly four components and they all pass through the vertex, by the log canonical threshold assumption, they must each pass through the vertex with multiplicity 1.  By Lemma \ref{remainders}, they must each have degree $1 \pmod 4$, but the degrees must sum to 14, which is impossible.  Similarly, if there are three components, at least two must have degree $1 \pmod 4$ and the third may have degree 1 or 2 $\pmod 4$, which in both cases is impossible to sum to 14.  If there are two components, they could have multiplicities $(1,1)$, $(1,2)$, $(1,3)$, or $(2,2)$ at the vertex, but the only case that could possibly sum to 14 is if both curves have degree 1 $\pmod 4$, so the curves are either a degree 1 and degree 13 curve or a degree 5 and degree 9 curve.  However, any smooth curve of degree at least 6 in $\PP(1,1,4)$ is not rational, so the larger degree component therefore must be singular at the vertex.  In the first case, the degree 13 curve has equation $\sum_{i = 0}^3 f_{13-4i}(x,y)z^i = 0$, where $x,y,z$ are the weighted coordinates on $\PP(1,1,4)$, and $f_j(x,y)$ denotes a degree $j$ homogeneous polynomial in $x$ and $y$.  To be singular at the vertex $[0:0:1]$, the term $f_{1}(x,y)$ must vanish, so in fact this curve has multiplicity at least 5 at the vertex, contradicting our assumption on multiplicity.  Similarly, the degree 9 curve in the second case has multiplicity at least 5 at the vertex, also a contradiction.
\end{proof}

Now, we have shown that any reduced curve $\Dcyo$ has exactly one component, is rational, and can have at most one singular point.  We will use classification results for rational cuspidal plane curves to prove no such curves exist with log canonical threshold at least $\frac{3}{7}$.  First, assume $\Dcyo$ is a plane curve. 

\begin{lemma}\label{oneredcom}
If $C\subset\PP^2$ is a reduced and irreducible degree $7$ rational curve with a single cuspidal singularity at $p \in C$ then analytic locally at $p$, $C$ is parametrized by $t\mapsto (t^6,t^7)$.  In particular, $\lct(\PP^2,C) < \frac{3}{7}$. 
\end{lemma}

\begin{proof}
The $\delta$ invariant of such a cusp is 15 (the genus of a smooth degree 7 plane curve). Thus according to Remark \ref{rmk:boundonnewtonpairs}, the number of Newton pairs of the cusp is at most 2.  When $k=1$ then by \cite[Thm. 1.1]{Unicuspidalcurves} the only possibility is $(a,b)=(6,7)$, as desired in the statement of the lemma. When $k=2$, then by looking at the classification \cite[Thm. 1.1]{TKThesis} of unicuspidal rational curves with 2 Newton pairs there are no possibilities.
\end{proof}

\begin{remark}
    Lemma~\ref{oneredcom} implies that we could expand Table \ref{tableofcusps} to list all unicuspidal rational curves of degree $\le 7$ by adding the following row: 

\begin{center}
\begin{table}[H] 
\scalebox{.9}{
\begin{tabular}{c|c|c|c|c|c}
    Degree & $\begin{array}{c}\text{Parameterization} \end{array}$& $\begin{array}{c}\text{Local equation}\\\text{of cusp} \end{array}$ & $\begin{array}{c}\text{Multiplicity}\\\text{sequence} \end{array}$ & $\begin{array}{c}\text{Newton}\\\text{pairs} \end{array}$ & $\begin{array}{c}\text{Log canonical}\\\text{threshold} \end{array}$  \\\hline
    7 & $ (x,y) =(t^6,t^7)$ & $y^6 = x^7$ & $(6)$ & $(6,7)$ & $13/42$ \\\hline 
\end{tabular}
}
\end{table}
\end{center}

\vspace{-.3in}

\end{remark}

Now, to complete the proof of Proposition \ref{reducedimpliessmooth}, it suffices to consider the case that $\Dcyo$ is a curve in $\PP(1,1,4)$. 

\begin{proposition}\label{oneredcomp114}
If $C\subset \PP(1,1,4)$ is a reduced and irreducible degree 14 rational curve with an isolated unibranch singularity at the vertex, then $\lct(\PP(1,1,4),C) < 3/7$.
\end{proposition}

\begin{proof}
Suppose that $C$ is as in the hypothesis of the proposition. Assume for contradiction that $\lct(\PP(1,1,4),C) \ge \frac{3}{7}$. Let
\[
\pi\cl \FF_4\ra \PP(1,1,4)
\]
be the minimal resolution of $\PP(1,1,4)$ with exceptional divisor $E\subset \FF_4$. Then
\[
\pi^*K_{\PP(1,1,4)} = K_{\FF_4}+\frac{1}{2}E.
\]
Let $\Ct$ be the strict transform of $C$ in $\FF_4$.  Because $4C$ is Cartier, we have 
\[ 
\pi^*C = \Ct + \frac{a}{4}E
\]
for $a \in \mathbb{Z}^+$.  The assumption that $\lct(\PP^2,C)\ge 3/7$ implies that 
\[ \frac{1}{2} + \frac{3a}{28} \le 1\]
so $a \le 4$.  Finally, intersecting $\pi^*C = \Ct + \frac{a}{4}E$ with a fiber $\ell$ of the ruled surface $\FF_4$ gives 
\[ \frac{14}{4} = \Ct\cdot \ell + \frac{a}{4} \]
and $\Ct \cdot \ell \in \mathbb{Z}$ implies that $a = 2$.  We can use this to compute the arithmetic genus of $\Ct$ is 14: 
\[ 
2g_a(\Ct) - 2 = (K_{\FF_4} + \Ct) \cdot \Ct = (K_{\PP(1,1,4)} + C)\cdot C - E \cdot \Ct = 28 - 2 = 26.
\]
Alternatively, using the basis $\langle E, \ell \rangle$ for $\Pic(\FF_4)$, we can compute $\Ct \in |3E + 14 \ell|$ to determine the genus.  

By assumption, $\Ct$ is then a rational curve with a single unibranch singularity with $\delta$-invariant 14 at the unique intersection point $p\in \Ct \cap E$. Moreover $(\FF_4,\frac{3}{7}\Ct)$ is log canonical.  By Remark \ref{rmk:boundonnewtonpairs}, the cusp is parameterized by $k$ Newton pairs with $k \le 2$.  If $k=2$ then there are two Newton pairs with $M_1=a\ge 4$, $b_1>a$, and $\gcd(a,b_1)\ge 2$, which implies $b_1\ge 6$. Thus
\[
\lct(\FF_4,\Ct)\le\frac{1}{4}+\frac{1}{6} <\frac{3}{7},
\]
a contradiction.

Therefore, there can only be one Newton pair $(M_1,N_1) = (a,b)$. In this case, $28 = (M_1 -1)(N_1 -1) =(a-1)(b-1)$.  This has three solutions: $(a,b) = (5,8)$, $ (3,15)$ or $(2,29)$.  In the first two cases,
\[
\lct(\FF_4,\tilde{C}) = \frac{1}{5}+\frac{1}{8} < \frac{3}{7},
\]
or
\[
\lct(\FF_4,\Ct) = \frac{1}{3}+\frac{1}{15} < \frac{3}{7},
\]
which both give a contradiction as above.

It remains to show that there is no rational curve $\Ct\subset \PP^4$ in the linear system $|3E+14\ell|$ with a single unibranch singularity along $E$ having analytic local equation $y^2=x^{29}$. We rule out this last possibility by transforming this curve to a cuspidal curve in $\PP^2$ and applying the work of Borodzik and Livingston. This is carried out in the following two lemmas.
\end{proof}

\begin{lemma}
If there is a curve $\Ct\subset \FF_4$ as above, then there is an irreducible rational degree 7 curve $\Gamma\subset \PP^2$ with two cusps at points $p_1,p_2\in \PP^2$ having each having a single Newton pair: (2,19) at $p_1$ and $(4,5)$ at $p_2.$
\end{lemma}

\begin{proof}
Let
\[
\pi: (X,D) \to (\mathbb{F}_4, \Ct)
\]
be the minimal embedded resolution (with $D:= \pi^*(\Ct)^\red$).  From above, $\Ct$ has a cusp with Newton pair $(2,29)$ at its unique set-theoretic intersection point $p=E\cap \Ct$. Let $F_p\in |\ell|$ be the fiber going through $p$.

On $X$, $D$ is a rational tree consisting of the strict transform $\Ct_X$ of $\Ct$ and 16 exceptional divisors $E_1,\cdots,E_{16}$ (shown in the graph below).

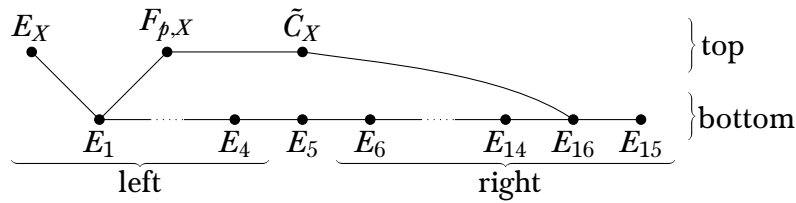
\begin{center}
\begin{tikzpicture}[scale=.9]
\draw (0,1) -- (1,0);
\draw (7,0) -- (8,0);
\draw (8,0) -- (9,0);
\draw (8,0) .. controls (7, 0.7) and (5, .85) .. (4,1);
\draw (1,0) -- (2,1);
\draw (2,1) -- (4,1);
\filldraw (1,0) circle (2pt) node[below] {$E_1$};
\draw (1,0) -- (1.75,0); 
\draw[dotted] (1.75,0) -- (2.25,0); 
\draw (2.25,0) -- (3,0);
\foreach \x in {4,5,6}
    \filldraw (\x-1,0) circle (2pt) node[below] {$E_{\x}$};
\draw (5,0) -- (5.75,0); 
\draw[dotted] (5.75,0) -- (6.25,0); 
\draw (6.25,0) -- (7,0);
\filldraw (7,0) circle (2pt) node[below] {$E_{14}$};
\filldraw (8,0) circle (2pt) node[below] {$E_{16}$}
        (9,0) circle (2pt) node[below] {$E_{15}$};
\draw (3,0) -- (5,0);
\filldraw (0,1) circle (2pt) node[above] {$E_X$}
        (4,1) circle (2pt) node[above] {$\tilde{C}_X$}
        (2,1) circle (2pt) node [above] {$F_{p,X}$};

\draw [decorate,
    decoration = {calligraphic brace}] (9.7,.4) -- (9.7,-.3) node [below right] (R){};
\path (R) ++(.7,.5) node {bottom};
\draw [decorate,
    decoration = {calligraphic brace}] (9.7,1.5) -- (9.7,.7) node [below right] (R){};
\path (R) ++(.35,.5) node {top};
\draw [decorate,
    decoration = {calligraphic brace}] (3.5,-.6) -- (-.3,-.6) node (R){};
\path (R) ++(1.9,-.3) node {left};
\draw [decorate,
    decoration = {calligraphic brace}] (9.5,-.6) -- (4.5,-.6) node [below right] (R){};
\path (R) ++(2.4,-.2) node {right};
\end{tikzpicture}
\end{center}

\noindent On $X$, $E_i^2 = -2$ for $1\le i \le 15$ and $E_{16}^2=-1$. $E_X$ (resp. $F_{p,X}$) is the strict transform of $E$ (resp. $F_p$). We also have $E_X^2=-5$ and $F_{p,X}^2=-1$.

The pair $(\FF_4,C)$ is obtained by contracting the bottom curves of the dual graph. It is also possible to simultaneously contract the left and the right curves. The result of this contraction is a smooth rational surface of Picard rank 1, so must be $\PP^2$. Let $\Gamma \subset \PP^2$ be the image of $\Ct_X$. From the description of the dual graph, this produces two unibranch singularities on $\Gamma$ with Newton pairs $(2,19)$ and $(4,5)$. (Note that the surface $X$ is not a minimal resolution of $(\PP^2,\Gamma)$; to obtain the minimal resolution, we must first contract $F_{p,X}$, and at this point the dual graph of the exceptional locus uniquely determines the singularity type.) Now, the image of $E_5$ is a line in $\PP^2$, which together with the above dual graph can be used to show that $\Gamma$ has degree 7 as desired.
\end{proof}

To complete the proof of the proposition, we show that $\Gamma$ does not exist.

\begin{lemma}
There is no degree $7$ rational curve $\Gamma\subset \PP^2$ with two cuspidal singularities, each with a single Newton pair of types $(2,19)$ and $(4,5)$. 
\end{lemma}

\begin{proof}
Suppose that such a curve exists, and let $p_1$ be the $(2,19)$ cusp and $p_2$ the $(4,5)$ cusp. The associated semigroups (Definition \ref{defn:semigroup}) are: 
\[
W_{p_1} = \{ 0 , 2, 4, 6, 8, 10, 12, 14, 16,  \dots \} \text{ and }W_{p_2} = \{ 0 , 4 , 5, 8, 9, 10, 12, 13, 14, 15, \dots \}.
\] 
Now we apply Equation (\ref{countingfunction}) when $d = 7$, $n = 2$, and $j = 2$. This reads 
\begin{align}
\min_{\substack{k_1, k_d \in \mathbb{Z}; \\ k_1 +  k_2 = 15}} ( R_{p_1}(k_1) + R_{p_2}(k_2)) = 6.  \qquad
\end{align}
Here recall that $R_{p_i}(k):= \#W_{p_i}\cap [0,k)$. 

Let $R(k_1, k_2) = R_{p_1}(k_1) + R_{p_2}(k_2)$. In the case $k_1 \ge 13$ (so $k_2 \le 2)$, then $R_{p_1}(k_1)\ge 7$, so $R(k_1, k_2) \ge 7$.  Similarly, in the case $k_2 \ge 13$, then $R_{p_2}(k_2) \ge 7$, so $R(k_1, k_2) \ge 7 > 6$.  Checking all intermediate values:
\begin{center}
\begin{tabular}{ccccc}
$R(12, 3) = 7$& $R(11,4) = 7$& $R(10,5) = 7$& $R(9,6) = 8$& $R(8,7) = 7$\\
$R(7,8) = 7$& $R(6,9) = 7$& $R(5,10) = 8$& $R(4,11) = 8$& $R(3,12) = 8$
\end{tabular}
\end{center}
proves that $\Gamma$ does not exist.  
\end{proof}

\begin{remark}\label{unicuspidal}
    The previous arguments rely heavily on the classification of rational unicuspidal plane curves for $d \le 6$ in Table \ref{tableofcusps} and for degree $d = 7$ in Lemma \ref{oneredcom}.  To generalize the arguments in this section to larger primes $p$, one at least needs a bound on the log canonical threshold of unicuspidal plane curves of degree $d \le p$.  For degree $\le 7$, we see that the log canonical threshold of the curve is always less than $\frac{3}{d}$.  Note that this is not always the case: Orevkov has exhibited two sporadic curves rational unicuspidal curves with a single Newton pair, one of degree 8 and one of degree 16, with log canonical threshold larger than $\frac{3}{d}$ (\cite{Orevkov}, or \cite[Thm. 1.1(e)(f)]{Unicuspidalcurves}).  However, these are highly special: Orevkov conjectures that these are the only examples of unicuspidal rational curves with a single Newton pair with a particularly large degree as compared to the multiplicity \cite[Pg. 2]{Orevkov}.  To that end, we conclude this section with a classification question.  
\end{remark}

\noindent\textbf{Question.} \textit{For what degrees $d$ does there exist a unicuspidal rational plane curve of degree $d$ with log canonical threshold at least $\frac{3}{d}$?  Of particular interest is the case when $d = p$ is a prime number.}

\section{On Hacking's Calabi-Yau limits of septic plane curves that are nonreduced}\label{sec:nonreduceddeg7}

We use the results of the previous section and several additional computations to prove Theorem \ref{thmE:deg7}.  By \S\ref{sec:reducedlimits}, we only need to consider the case that $\Dcyo$ is non-reduced.  Throughout this section we assume for contradiction that there is a nonreduced component of $\Dcyo$.

\begin{proof}[Proof of Theorem E]
By Proposition \ref{reducedimpliessmooth}, we only need to consider the case that there is an nonreduced component of $\Dcyo$.  The following table enumerates the possible degrees of reduced and nonreduced components of $\Dcyo$.

\[
\begin{array}{cc}
\text{Curves in $\PP^2$} & \text{Curves in $\PP(1,1,4)$}\\
&\\
\begin{array}[t]{c|c|l}
\begin{array}{c}\text{Reduced}\\\text{degrees}\end{array}&\begin{array}{c}\text{Nonreduced}\\\text{degrees}\end{array}&\text{Proof}\\\hline
5&1&\text{Lem. \ref{nonredtwocomps}}\\\hline
3&2&\text{Lem. \ref{nonredtwocomps}}\\\hline
3&1+1&\text{Lem. \ref{twononredonered}}\\\hline
1&3&\text{Prop. \ref{prop:worstdeg7curve}}\\\hline
1&2+1&\text{Lem. \ref{twononredonered}}\\\hline
4+1&1&\text{Lem. \ref{twononredonered}}\\\hline
3+2&1&\text{Lem. \ref{onenoneredtwored}}\\\hline
2+1&2&\text{Lem. \ref{onenoneredtwored}}\\\hline
\end{array}& \begin{array}[t]{c|c|c}\begin{array}{c}\text{Reduced}\\\text{degrees}\end{array}&\begin{array}{c}\text{Nonreduced}\\\text{degrees}\end{array}&\text{Proof}\\\hline
12&1&\text{Lem. \ref{nonredtwocomps}}\\\hline
6&4&\text{Lem. \ref{nonredtwocomps}}\\\hline
4&5&\text{Lem. \ref{nonredtwocomps}}\\\hline
4&4+1&\text{Lem. \ref{twononredonered}}\\\hline
8+4&1&\text{Lem. \ref{twononredonered}}\\\hline
\end{array}
\end{array}
\]

In each case, we prove that the configuration gives a contradiction.  The computations are carried out in the lemmas and propositions indicated in the table. 
\end{proof}

\begin{lemma}\label{nonredtwocomps}
It is not possible that $\Dcyo=R+2N\subset \Xcyo$ is non-reduced and has two components, except for possibly in the case of a line and a doubled cubic curve.
\end{lemma}

\begin{proof}
Suppose for contradiction that there is such a limit that is not a line and a doubled cubic. Write:
\[
\Dcyo = R+2N
\]
where $R$ is the reduced curve and $N$ is the non-reduced curve. By Lemma~\ref{lem:Dnormsatisfies(Hi)} and Theorem~\ref{thm:possibledualgraphs}, $R$ and $N$ meet at a single point, and if there are any singularities of $R$ and $N$ they must be unibranch. By degree considerations (in the previous table) either $R$ or $N$ is smooth: one of them is either irreducible of degree 1 or 2 in $\PP^2$ or irreducible of degree 1 or 4 in $\PP(1,1,4)$. By Lemma~\ref{lctcomputationfortwocomps}, if the intersection of $R$ and $N$ has length at least 4 then $\lct(\Xcyo,\Dcyo)$ is too small. There is only one remaining case to consider: $\Dcyo\subset \PP(1,1,4)$ and $R$ and $N$ have degrees 12 and 1 respectively.

In this case, the degree 12 curve $R$ must be singular and rational with cuspidal singularities. The $\delta$-invariant of any cusp on $R$ is the genus of a smooth degree 12 curve in $\PP(1,1,4)$ which is 10. By Lemma~\ref{deltabranches}, such a cusp cannot lie in the normal locus of $\Dcy$. Therefore, there is exactly one cusp at the unique intersection point of $R$ and $N$. By multiplicity considerations, we see that the cusp in $R$ is a double point. Note: $R$ and $N$ do not meet at the vertex in $\PP(1,1,4)$: $R$ is Cartier and any Cartier divisor meeting the vertex has multiplicity at least 4, so $R+2N$ would be too singular.

Therefore, $R$ and $N$ meet at a smooth point in $\PP(1,1,4)$. Blowing up this smooth point, the strict transforms of $R$ and $N$ still intersect (as they met to length 3 on $\PP(1,1,4)$). Blowing up this new intersection point gives an exceptional divisor with discrepancy $3/8$. Thus $\lct(\PP(1,1,4),R+2N)\le 3/8<3/7$ a contradiction.
\end{proof}

\begin{lemma}\label{twononredonered} It is not possible for $\Dcyo$ to be the union of a reduced curve and two non-reduced curves.
\end{lemma}

\begin{proof}
Suppose for contradiction that
\[
\Dcyo = R+2N_1+2N_2.
\]
By Lemma \ref{lem:Dnormsatisfies(Hi)} and Theorem \ref{thm:possibledualgraphs}, $\Dcyo$ has the following intersection graph.

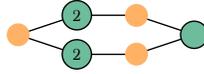
\begin{figure}[h]
\centering
\scalebox{1.3}{\begin{tikzpicture}[gren0/.style = {draw, circle,fill=greener!80,scale=.7},gren/.style ={draw, circle, fill=greener!80,scale=.4}] 
\node[scale=.5] at (0,0) (1){};
\node[gren] at (.6,-.2) (2){$2$};
\node[gren] at (.6,.2) (3){$2$};
\node[scale=.5] at (1.2,-.2) (4){};
\node[scale=.5] at (1.2,.2) (5){};
\node[gren0] at (1.8,0) (6){};
\draw (1)--(2);
\draw (1)--(3);
\draw (2)--(4);
\draw (3)--(5);
\draw (4)--(6);
\draw (5)--(6);
\filldraw [orange!60,scale=.11] (11,1.8) circle () {};
\filldraw [orange!60,scale=.11] (11,-1.8) circle () {};
\filldraw [orange!60,scale=.11] (0,0) circle () {};
\end{tikzpicture}}
\end{figure}

\noindent There are 3 cases.
\begin{enumerate}
\item$\Dcyo\subset \PP^2$, $R$ has degree 3, and $N_1$ and $N_2$ both have degree 1. $R$ must be a cuspidal rational curve. The cusp must lie on the line $N_1$ or $N_2$ and the length of the intersection must be 3. By Lemma~\ref{lctcomputationfortwocomps}, $\lct(\PP^2, R+2N_1+2N_2)\le 5/12<3/7$, a contradiction.
\item $\Dcyo\subset \PP^2$, $R$ has degree 1, $N_1$ has degree 2, and $N_2$ has degree 1. The doubled conic $N_1$ and the doubled line $N_2$ are tangent at a point. By Lemma \ref{lctcomputationfortwocomps}, $\lct(\PP^2,2N_1+2N_2)$ is at most $3/8$, a contradiction.
\item $\Dcyo \subset \PP(1,1,4)$, $R$ has degree 4, $N_1$ has degree 4 and $N_2$ has degree 1. $R$ and $N_1$ are both smooth and meet at a single point to length 4. Then, by Lemma \ref{lctcomputationfortwocomps}, $R+2N_1$ has log canonical threshold at most $5/12$, a contradiction. \qedhere
\end{enumerate}
\end{proof}

\begin{lemma}\label{onenoneredtwored}
It is not possible that $\Dcyo$ has two reduced components and 1 nonreduced component.
\end{lemma}

\begin{proof}
Suppose for contradiction that
\[
\Dcyo=R_1+R_2+2N.
\]
By Lemma \ref{lem:Dnormsatisfies(Hi)} and Theorem \ref{thm:possibledualgraphs}, $\Dcyo$ has the following intersection graph.
\begin{figure}[h]
\centering
\scalebox{1.3}{\begin{tikzpicture}[gren0/.style = {draw, circle,fill=greener!80,scale=.7},gren/.style ={draw, circle, fill=greener!80,scale=.4}] 
\node[scale=.5] (1) at (0,0) (1){};
\node[gren] at (-.6,0) (2){$2$};
\node[gren0] at (.6,-.2) (3){};
\node[gren0] at (.6,.2) (4){};
\draw (1)--(3);
\draw (1)--(2);
\draw (1)--(4);
\filldraw [orange!60,scale=.11] (0,0) circle () {};
\end{tikzpicture}}
\end{figure}
There are four cases:
\begin{enumerate}
\item $\Dcyo\subset \PP^2$ and $R_1$, $R_2$, $N$ have degrees 4, 1, and 1 respectively.
\item $\Dcyo\subset \PP^2$ and $R_1$, $R_2$, $N$ have degrees 3, 2, and 1 respectively.
\item $\Dcyo\subset \PP^2$ and $R_1$, $R_2$, $N$ have degrees 2, 1, and 1 respectively.
\item $\Dcyo\subset \PP(1,1,4)$ and $R_1$, $R_2$, $N$ have degrees 8, 4, and 1 respectively.
\end{enumerate}
All of these curves must be rational and everywhere unibranch. Any singularity must occur at the unique intersection point with $N$. In cases (1), (2), and (4) this shows the unique intersection point has multiplicity 5 which is too large by (H4). In case (3), this would force that $R_1$, $R_2$, and $N$ only meet at a single point $p$, but this means that $R_2$ and $N$ are both tangent lines to the conic $R_1$, which is impossible as the tangent line is unique. 
\end{proof}

The final case to eliminate is the possibility of a reduced line and a non-reduced (doubled) cubic curve in $\PP^2$. 

\begin{proposition}\label{prop:worstdeg7curve}
The curve $\Dcyo\subset \PP^2$ is not the union of a doubled cubic and a line.
\end{proposition}

We will use a series of lemmas in the proof of the proposition.  

\begin{lemma}
    If $\Dcyo$ consists of a double cubic curve and a line, then the cubic is smooth, the line is an inflection line, and the smooth limit $D_0$ must be the component of $\Dnormo$ that maps to $N$.
\end{lemma}

\begin{proof}
Suppose that $\Dcyo = R + 2N$, where $R$ is a line and $N$ is a cubic.  By Theorem~\ref{thm:possibledualgraphs}, $N$ cannot be nodal, and $N$ cannot be cuspidal as the log canonical threshold of a double (non-reduced) curve with a cusp is $5/12<3/7$.  Therefore, $N$ must be smooth. Theorem~\ref{thm:possibledualgraphs} also implies $R$ and $N$ meet at a single point with length 3, so $R$ is an inflection line.  The smooth limit $D_0$ must be the component of $\Dnormo$ that maps to $N$ because any other component of $\Dnormo$ is necessarily rational. 
\end{proof}

\begin{lemma}\label{lem:4/9canmodel}
If $\Dcyo \subset \PP^2$ is a doubled smooth cubic and an inflection line, there is a suitable weighted-blow up of the intersection point of the curves in the family $(\Xcy, \Dcy)$ that improves the singularity: either the strict transform of $\Dcy$ in the exceptional divisor is reduced, or consists of two curves (one of which is doubled) meeting to order at most 2. 
\end{lemma}

\begin{proof}
Suppose that $(\Xcyo, \Dcyo) = (\PP^2, R+2N)$, where $R$ is a line in $\PP^2$ and $N$ is a smooth cubic meeting $R$ at one point to with length 3.  In this case, the central fiber of the normalization $\Dnormo$ necessarily contains 2 curves: the curve whose image in $(X,D)$ is the smooth curve $D_0$ (which double covers the doubled elliptic curve, and by abuse of notation which will denote also by $D_0$) and the strict transform of the line. The main idea in this proof is that the singularity at the intersection point of $R$ and $N$ has the worst log canonical threshold ($4/9$), and that a suitable weighted blow-up improves the situation. We then analyze the resulting modification of $\Dcy$.

Analytic locally at the intersection point, $\Dcyo$ has equation $x^2(x-y^3)=0$. If $t$ is an analytic coordinate vanishing at $0\in T$, then the total space $(\Xcy,\Dcy)$ is a hypersurface and near the singular point on the central fiber it has analytic local equation
\[
G(x,y,t) = x^2(x-y^3) + \sum_{n \ge 1} t^n g_n(x,y) = 0.
\]
Making a weighted blow-up depends on the coordinate system, so to start we want to ensure we are working in an ideal coordinate system for the equation $G$. To a system of coordinates $x$ and $y$ we associate the number
\[
\zeta(x,y) := \max\left\{\frac{9-3i-j}{n} \middle| \begin{array}{l}
n\ne 0,\text{ and the monomial }x^iy^jt^n\text{ appear}\\
\text{with nonzero coefficient in }G(x,y,t) \end{array}\right\}.
\]
Now, for any $\epsilon>0$ there are only finitely many monomials $x^iy^jt^n$ such that $(9-3i-j)/n>\epsilon.$ On the other hand, there is a positive lower bound on $\zeta(x,y)$ that is independent of the change of coordinates
\[
\xbar = x+t h_1(x,y)\text{ and }\ybar = y+th_2(x,y)
\]
given by $1/N$ where $N$ is any power of $t$ that annihilates the cokernel of the relative normal sequence
\[
T_{\PP^2_T/T}|_\Dcy\ra \Oc_\Dcy(\Dcy).
\]
Here we are using that $\Dcy$ is a smoothing of $\Dcyo$.  Therefore, we may assume that there is no change of coordinates $\xbar,\ybar$ such that $\zeta(\xbar,\ybar)<\zeta(x,y)$.

Define $k$ to be the minimal power of $t$ such that there is a monomial $x^iy^jt^k$ that achieves the maximum in $\zeta(x,y)$. Let $w = 3i+j$ (necessarily $\le 8$) and consider the base change that takes a $(9-w)$th root of $t$, i.e. set $s^{(9-w)} = t$. The base change of $\Dcy$ has analytic local equation: $G(x,y,s^{(9-w)})=0$.
Let
\[
\mu\cl Y\ra X
\]
be the weighted blow-up of the base change of $\Xcy$ where the coordinates $(x,y,s)$ have weights $(3k,k,1)$, and let $D_Y$ be the strict transform of $\Dcy$. In these weights, we can compute:
\[
\wt(s^{(9-w)n}x^iy^j) = 9n - nw +k(3i+j).
\]
Thus the minimal possible weight is $9k$. So the exceptional divisor of $\mu$ is a degree $9k$ curve in $\PP(3k,k,1)$. There is an isomorphism of the exceptional divisors $\PP(3k,k,1)$ with $\PP(3,1,1)$ (this is the cone over the twisted cubic) given by $[x:y:s]\mapsto [x:y:s^k]$. This is not to say that $Y$ is isomorphic to a (3,1,1) weighted blow-up. Generically $Y$ has $A_{k-1}$-surface singularities along the locus $(s=0)\subset \PP(3k,k,1)\subset Y$.

The surface $D_Y$ is S2 and birational to $\Dcy$. The curve $D_Y\cap \PP(3,1,1)\subset \PP(3,1,1)$ is a degree 9 curve. Every component of $D_Y\cap \PP(3,1,1)$ is rational and the curve when thought of in $\PP(3k,k,1)$ is defined by the lowest weight monomials in $G(x,y,s^{(9-w)})$. The intersection with the line $(s=0)\subset \PP(3k,k,1)$ is $x^2(x-y^3)=0$. Thus there is at most one nonreduced component of $D_Y\cap \PP(3k,k,1)$ of multiplicity at most 2. Moreover there are no components with degree not divisible by 3, as the curve does not meet the cone point of $\PP(3,1,1)$. Finally, we have that the intersection graph of the normalization $(D_Y^\norm)_0$ is still a tree.

Now we show that our assumption that the coordinates $(x,y)$ minimize $\zeta(x,y)$ implies that $D_Y\cap \PP(3,1,1)\ne 2C_1+C_2$ where $C_1$ and $C_2$ are smooth, degree 3 rational curves in $\PP(3,1,1)$ that meet to length 3. Suppose to the contrary that $D_Y$ has these two exceptional components.  Denote by $[x:y:v]$ the coordinates on $\PP(3,1,1)$ (so $v = s^k$ in the isomorphism $\PP(3k,k,1) \cong \PP(3,1,1)$).  Then the curves have equations:
\[
C_1 = \left(x+s f_2(y,v)=0\right), C_2 = \left(x+(y+\lambda v)^3 = 0\right)\subset \PP(3,1,1)
\]
where $f_2(y,v)$ is a quadratic polynomial in $y$ and $v$. The condition that $C_1$ and $C_2$ meet to order 3 implies (after substitution) that
\begin{equation}
v f_2 + (-y-\lambda v)^3 = (ay+bv)^3.
\end{equation}
Now, if $v$ divides $f_2$ then $f_2=0$ (considering both sides modulo $v^2$ shows that the first two coefficients of the expanded cubes are the same, which implies they are equal). So either $f_2=0$ or $v$ does not divide $f_2.$

Lifting the equation for $C$ to the exceptional divisor $\PP(3k,k,1)$ under the isomorphism $s \mapsto s^k = v$ shows $D_Y \cap \PP(3k,k,1)$ has equation 
\[
(x+s^kf_2(y,s^k))^2(x+(y+\lambda s^k)^3) = 0
\]
Expanding this equation, all $s$ exponents are divisible by $k$. On the other hand, this equation is the tangent cone of an equation pulled back under the map $s\mapsto s^{(9-w)}$ so all the $s$ exponents of the expansion are divisible by $(9-w)$. After expanding the coefficient of $x$ is
\[
2s^k(y+\lambda s^k)^3f_2(y,s^k) + s^{2k}f_2^2
\]
If $f_2\ne 0$ then $s$ divides the right hand side to order exactly $k$, thus $k=(9-w)\ell$ for some $\ell>0$. Similarly if $f_2=0$ then analyzing the $x^2$ term shows that $k=(9-w)\ell$ for some $\ell>0$. As a consequence, the change of coordinates
\[
\xbar = x+s^ks_2(y,s^k),\quad \ybar = y+\lambda s^k
\]
lifts to the change of coordinates:
\[
\xbar = x+t^\ell s_2(y,t^\ell), \quad \ybar = y+\lambda t^\ell
\]
prior to the base change.

Now we claim that this gives $\zeta(\xbar,\ybar)< \zeta(x,y)$, giving the desired contradiction. Suppose that the monomial $x^iy^jt^n$ appears with nonzero coefficient in $G(x,y,t).$ We need to study the monomials that appear in the expansion of
\[
(\xbar-t^\ell s_2(\ybar-\lambda t^\ell,t^\ell))^i(\ybar-\lambda t^\ell)^jt^n.
\]
Observe that if $(\xbar,\ybar,t)$ are given weights $(3\ell,\ell,1)$ then every monomial that appears has weight $\ell(3i+j)+n$. Suppose $\xbar^{i'}\ybar^{i'}t^{n'}$ appears in this expansion. Now:
\begin{align*}
(9-w_n)/n\le(9-w)/k=1/\ell &\iff (9-w_n)\ell \le n\\
\iff (9-w')\ell+(w'-w_n)\ell \le n &\iff (9-w')\ell\le n+\ell(w_n-w')\\
\iff (9-w')\ell \le (9-w)(n+\ell(w_n-w'))&\iff (9-w')/n'\le 1/\ell=(9-w)/k.
\end{align*}
So we know $\zeta(\xbar,\ybar)\le \zeta(x,y)$. The only terms that can achieve equality in the above inequalities are the monomials $x^iy^jt^n$ such that $(9-3i-j)/n = (9-w)/k$, which are exactly the lowest weight monomials with respect to the $(3\ell,\ell,1)$-weighting. But this change of coordinates has been chosen so that after the change of coordinates, the lowest weight part of $G(\xbar,\ybar,t)$ is exactly $\xbar^2(\xbar+\ybar^3)$. Thus $\zeta(\xbar,\ybar,t)<\zeta(x,y,t)$ --- a contradiction.

We have now shown that there is a weighted blow-up improving the worst singularity (the point of intersection of the non-reduced cubic and inflection line), and this results in two cases to consider: first, the curve in the exceptional divisor $\PP(3,1,1)$ is reduced, or that the curve in $\PP(3,1,1)$ is non-reduced.  In this case, as the curve avoids the singular point of $\PP(3,1,1)$, it is necessarily the union of a doubled degree 3 curve and another degree 3 curve, which can meet to order at most 2 at any point. 
\end{proof}

\begin{lemma}\label{lem:reducedcasepart1}
In the weighted blow-up and notation from Lemma \ref{lem:4/9canmodel}, if the intersection of the strict transform of $\Dcy$ with the exceptional divisor $\PP(3,1,1)$ is reduced, it must be a degree 9 curve with a unicuspidal singularity locally of the form $x^2 = v^{15}$. 
\end{lemma}

\begin{proof}
We continue to use the notation from Lemma \ref{lem:4/9canmodel} and its proof.  Suppose the curve $D_Y \cap \PP(3,1,1) \subset \PP(3,1,1)$ is reduced.  By the observations above, this curve has degree 9 in $\PP(3,1,1)$ and all components are rational.  The curve $C := D_Y \cap \PP(3,1,1)$ must intersect the line $(v = 0)$ to order 3 in two smooth points of the surface.

If $C$ is reducible, because every component has degree a multiple of 3 on $\PP(3,1,1)$, it has either two or three components. If the curve $C$ had two reduced components $C_1$ and $C_2$ of degrees 6 and 3, the degree 6 component must be rational (and hence singular), but cannot have any singularities away from $(v = 0)$ as they would have a non-zero contribution to the $\delta$-invariant of the curve, impossible by Remark \ref{rem:semistablereduction}.  Because one of the intersections of $C$ with $(v=0)$ is smooth with order 1, it cannot be a singular point of $C_1$, so in fact $C_1$ can have only one singularity, a double point on the line $(v = 0)$, which must be the singular point of $C_1$.  However, this implies that the intersections of the components $C_1$ and $C_2$ lie entirely away from the line $(v=0)$, and these intersections must be transverse because they have $\delta$-invariant 0 (Remark \ref{rem:semistablereduction}).  This is impossible as the dual graph is a tree.  For the same reasoning, it is impossible that the curve $C$ has three reduced degree 3 components.  

If the curve $C$ is irreducible, it is rational so must be singular, and by the argument above, can have only one singularity, a double point on the line $(v = 0)$.  In order for $C$ to be rational, this forces $C$ to have a singularity analytically locally of the form $x^2 = v^{15}$.  The reduced case will be concluded by the next lemma. 
\end{proof}

\begin{lemma}\label{lem:reducedcasepart2}
An isolated singularity of the form $x^2 = v^{15}$ does not exist on an irreducible degree 9 curve in $\PP(3,1,1)$.
\end{lemma}

\begin{proof}
Suppose such a curve $C$ exists.  We will transform this curve to a rational quintic curve on $\PP^2$ with a two singularities with Newton pairs $(2,3)$ and $(2,11)$.  First, resolve the $\frac{1}{3}(1,1)$ singularity on $\PP(3,1,1)$, which creates a $-3$ curve that does not intersect $C$.  Let $F$ be the fiber of the ruled surface $\mathbb{F}_3$ through the singular point of $C$, and blow up the singular point with exceptional divisor $E$, then contract the strict transform of $F$.  The resulting surface is $\mathbb{F}_2$, and the image of the curve $C$ has a $(2,13)$ singularity on its intersection with the image of $E$, and the image of $C$ meets the negative section of $\mathbb{F}_2$ transversally.  Now, blow up the singular point of the image of $C$ and contract the image of $E$.  We are now in $\mathbb{F}_1$, and the image of $C$ has a $(2,11)$ singularity and meets the negative section of $\mathbb{F}_1$ to order 2.  Contracting the negative section yields a curve in $\PP^2$ with a $(2,11)$ and $(2,3)$ singularity.  Because the curve was initially a degree 3 multi-section of $\mathbb{F}_3$ and met the $-1$ curve in $\mathbb{F}_1$ to order 2, the resulting curve in $\PP^2$ has degree 5.  However, by the classification of cuspidal quintic plane curves in $\PP^2$ in \cite[6.1.3]{MoeThesis}, this curve does not exist.  
\end{proof}

Now we are in the position to complete the proof of Proposition \ref{prop:worstdeg7curve}.

\begin{proof}
Suppose for contradiction that $(\Xcyo, \Dcyo) = (\PP^2, R+2N)$, where $R$ is a line in $\PP^2$ and $N$ is a smooth cubic meeting $R$ at one point to with length 3, and recall that the central fiber of the normalization $\Dnormo$ necessarily contains 2 curves: the curve whose image in $(X,D)$ is the smooth curve $D_0$ (which double covers the doubled elliptic curve, and by abuse of notation we will also denote by $D_0$) and the strict transform of the line.  By Lemma \ref{lem:4/9canmodel}, there is an appropriate weighted blow-up of the intersection point of $R$ and $N$ in the family $(\Xcy, \Dcy)$ with exceptional divisor isomorphic to $\PP(3,1,1)$ improving the singularity.  By Lemma \ref{lem:reducedcasepart1}, if the strict transform $D_Y$ in this weighted blow up has $C := D_Y\cap \PP(3,1,1)$ reduced, it must have a singularity analytically of the form $x^2=v^{15}$.  By Lemma \ref{lem:reducedcasepart2}, this is contradiction as this singularity does not exist.  Therefore, $C$ must be non-reduced.  Then, as all components have degree a multiple of three, $D_Y\cap \PP(3,1,1)$ is necessarily the union of two curves $C_1+ 2 C_2$ where $C_1$ and $C_2$ both have degree 3.  By Lemma \ref{lem:4/9canmodel}, the curves cannot meet at a single point to length 3 as this contradicts the choice of weighted blow-up.  If there are 3 points in the intersection then using the inequality in Theorem \ref{normalizationdualgraphproperties}, we obtain a contradiction to the fact that the intersection graph of $(D^\norm_Y)_0$ is a tree.

It remains to consider the case of 2 intersection points, in which case $C_1$ and $C_2$ must meet transversely at one point and to order two at the other point.

\begin{figure}[h]
A sketch of $(D_Y)_0$ and its intersection graph.
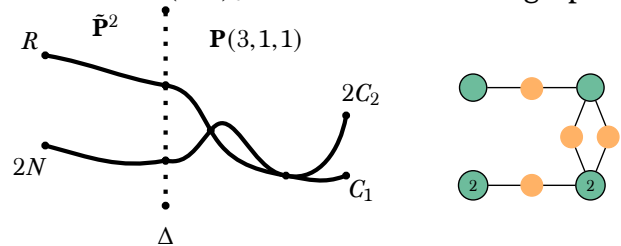
\begin{tabular}{cc}
{\scalebox{4}{\begin{tabular}{c}
\begin{tikzpicture}[gren0/.style = {draw, circle,fill=greener!80,scale=.7},gren/.style ={draw, circle, fill=greener!80,scale=.4},blk/.style ={draw, circle, fill=black!,scale=.03},lbl/.style ={scale=.2}] 
\node[blk] at (0,.35) (1){};
\node[blk] at (0,-.3) (2){};
\node[blk] at (-.4,-.1) (3){};
\node[blk] at (0,-.15) (4){};
\node[blk] at (-.4,.2) (5){};
\node[blk] at (0,.1) (6){};
\node[blk] at (0,.1) (7){};
\node[blk] at (.15,-.05) (8){};
\node[blk] at (0,-.15) (9){};
\node[blk] at (.4,-.2) (10){};
\node[blk] at (.6,-.2) (11){};
\node[blk] at (.6,0) (12){};
\node[lbl] at (.3,.25) (13){$\mathbf{P}(3,1,1)$};
\node[lbl] at (-.2,.3) (14){$\tilde{\mathbf{P}}^2$};
\node[lbl] at (0,-.4) (15){$\Delta$};
\node[lbl] at (-.45,-.17) (16){$2N$};
\node[lbl] at (-.45,.25) (17){$R$};
\node[lbl] at (.65,.07) (18){$2C_2$};
\node[lbl] at (.65,-.25) (19){$C_1$};
\draw [densely dotted] (1) to [out=-90,in=90] (2);
\draw [-] (3) to [out=-15,in=190] (4);
\draw [-] (5) to [out=-10,in=170] (6);
\draw [-] (7) to [out=-10,in=120] (8);
\draw [-] (9) to [out=-10,in=-130] (8);
\draw [-] (8) to [out=-60,in=170] (10);
\draw [-] (8) to [out=50,in=170] (10);
\draw [-] (10) to [out=-10,in=-160] (11);
\draw [-] (10) to [out=-10,in=-100] (12);
\end{tikzpicture}
\end{tabular}}} & {\hspace{-30pt}\begin{tabular}{c}\scalebox{1.3}{\begin{tabular}{c}
\begin{tikzpicture}[gren0/.style = {draw, circle,fill=greener!80,scale=.7},gren/.style ={draw, circle, fill=greener!80,scale=.4}]
\node[gren] at (-.6,0) (1){$2$};
\node[scale=.5] (1) at (0,0) (2){};
\node[gren] at (.6,0) (3){$2$};
\node[gren0] at (-.6,1) (4){};
\node[scale=.5] (1) at (0,1) (5){};
\node[gren0] at (.6,1) (6){};
\node[scale=.5] (1) at (.8,.5) (7){};
\node[scale=.5] (1) at (.4,.5) (8){};
\draw (1)--(2);
\draw (2)--(3);
\draw (4)--(5);
\draw (5)--(6);
\draw (6)--(7);
\draw (6)--(8);
\draw (3)--(7);
\draw (3)--(8);
\filldraw [orange!60,scale=.11] (0,0) circle () {};
\filldraw [orange!60,scale=.11] (0,9) circle () {};
\filldraw [orange!60,scale=.11] (3.7,4.5) circle () {};
\filldraw [orange!60,scale=.11] (7.1,4.5) circle () {};
\end{tikzpicture}
\end{tabular}}\\\vspace{10pt}\end{tabular}}
\end{tabular}
\end{figure}

As the intersection graph of $(D^\norm_Y)_0$ is a tree, it is necessary that the curve $C_2$ breaks into two components $C_\alpha$ and $C_\beta$ in the normalization. An easy analysis of the possible intersection graphs shows that $(D^\norm_Y)_0$ can be a tree only if one (and not both) of the curves $C_\alpha$ or $C_\beta$ meets $D_0$ in $(D^\norm_Y)_0$.

To reach a contradiction, we want to show that both branches $C_\alpha$ and $C_\beta$ both meet $D_0$ in $(D^\norm_Y)_0$. This can be checked analytically locally around the intersection point $p\in N\cap C_2$. Analytic locally $D_Y$ can be described as a $\mu_k$-quotient of a divisor in a simple normal crossing degeneration. More precisely, consider the threefold in affine space $\AA^4$ with coordinates $(x',y',z',t')$:
\[
Y_1=(y'z'=t')\subset \AA^4
\]
and consider the action of the $k$-th roots of unity $\mu_k=\langle \zeta \rangle$ that sends $y'
\mapsto \zeta y'$, $z'\mapsto (\zeta^{-1}) z'$, and fixes $x'$ and $t'$. Then there is an analytic local isomorphism $(Y_1/\mu_k,0)\cong (Y,p)$ (by abuse of notation we use $Y_1$ to denote an analytic neighborhood of $0$ that gives rise to such an isomorphism). Under the analytic map
\[
\rho\cl Y_1\ra Y
\]
we can assume that the equation of $\mu_k$-invariant divisor $D_1:=\rho^*D_Y$ is given by 
\[
D_1 = \left((x')^2+t'f(x',y'^k,z'^k,t')=0\right)\subset Y_1.
\]
Assume, without loss of generality that in this neighborhood $\rho$ maps the curve $(t'=x'=y'=0)$ dominantly onto $C_2$ and $(t'=x'=z'=0)$ onto $N$, and by construction, note that $(Y_1, \frac{1}{2}D_1 + (Y_1)_0)$ is log canonical.

Now we define $Y_{i+1}$ and $D_{i+1}$ iteratively as follows. Assume that there is only one curve in $D_{i}$ that dominates $(t'=x'=y'=0)\subset Y_1$ (i.e. the branches have not been separated yet). Then blow-up the reduced ideal of this curve in $Y_i$ to arrive at $\pi_{i+1}: Y_{i+1} \to Y_i$ and let $D_{i+1}$ be the strict transform of $D_i$.  Because $K_{Y_{i+1}} + \frac{1}{2}D_{i+1} + (Y_{i+1})_0 = \pi^*(K_{Y_i} + \frac{1}{2}D_i + (Y_i)_0)$, the pair $(Y_{i+1}, \frac{1}{2}D_{i+1} + (Y_{i+1})_0)$ is log canonical by \cite[Lem. 2.30]{KollarMori}. By induction $Y_i$ has a natural $\mu_k$-action, $D_i$ is preserved by this $\mu_k$-action, and the ideal of the blown-up curve has a natural equivariant structure. Thus $Y_{i+1}$ carries a natural $\mu_k$-action and $D_{i+1}$ is preserved by this action. After $\ell>1$ blow-ups, $D_\ell$ has two distinct branches that each map isomorphically onto $(t'=x'=y'=0)$. The natural map:
\[
Y_\ell/\mu_k\ra Y
\]
is birational onto the original analytic neighborhood of $p\in Y$. It follows that the pair $(Y_\ell, \frac{1}{2}D_\ell + (Y_\ell)_0)$ is log canonical and hence  $(Y_\ell/\mu_k,(\frac{1}{2}D_\ell+(Y_\ell)_0)/\mu_k)$ is log canonical by \cite[Prop. 5.20]{KollarMori}.  Therefore $D_{\ell}$ cannot contain any components of the double locus (non-normal locus) of $Y_\ell$ (otherwise, the pair would not be log canonical). Therefore, the quotient $D_\ell/\mu_k$ is an S2 partial normalization of $D_Y$ which separates $C_\alpha$ and $C_\beta$.  However, each blow-up has been the blow-up of a smooth curve in $Y_i$, so the total space remains $\QQ$-factorial, and hence both branches $C_\alpha$ and $C_\beta$ must intersect (the strict transform of) $N$.  Therefore, by Corollary \ref{cor:branchesmustmeetinnorm}, both branches $C_\alpha$ and $C_\beta$ intersect $D_0$ (the pre-image of $N$) in $(D^\norm_Y)_0$, and we have reached a contradiction. 
\end{proof}

\bibliographystyle{siam} 
\bibliography{smoothcurves.bib} 

\end{document}